\documentclass{article}%
\usepackage{amsmath}
\usepackage{amssymb}
\usepackage{graphicx}
\usepackage{hyperref}
\usepackage{amsfonts}%
\usepackage[usenames,dvipsnames]{color}
\providecommand{\U}[1]{\protect\rule{.1in}{.1in}}
\newtheorem{theorem}{Theorem}

\newtheorem{corollary}[theorem]{Corollary}

\newtheorem{definition}[theorem]{Definition}

\newtheorem{proposition}[theorem]{Proposition}
\newtheorem{remark}[theorem]{Remark}

\newenvironment{proof}[1][Proof]{\noindent\textbf{#1.} }{\ \rule{0.5em}{0.5em}}

\newcommand{\D}{\mathbb{D}}
\newcommand{\E}{\mathbb{E}}

\newcommand{\Q}{\mathbb{Q}}
\newcommand{\R}{\mathbb{R}}
\newcommand{\T}{\mathbb{T}}

\newcommand{\Z}{\mathbb{Z}}
\newcommand{\wh}{\widehat}

\newcommand{\ovl}{\overline}
\newcommand{\ep}{\epsilon}

\newcommand{\bk}{\mathbf{k}}

\begin{document}

\title{Noise based on vortex structures in 2D and 3D}

\author{Franco Flandoli\footnote{Scuola Normale Superiore di Pisa. Piazza Dei Cavalieri 7. Pisa PI 56126, Italy. E-mail: franco.flandoli@sns.it} \, and Ruojun Huang\footnote{Fachbereich Mathematik und Informatik, Universit\"at Münster. Einsteinstr. 62, M\"unster 48149, Germany. E-mail: ruojun.huang@uni-muenster.de}}

\maketitle

\begin{abstract}
A new noise, based on vortex structures in 2D (point vortices) and 3D (vortex
filaments), is introduced. It is defined as the scaling limit of a jump
process which explores vortex structures and it can be defined in any domain,
also with boundary. The link with Fractional Gaussian Fields and Kraichnan noise is discussed. The vortex
noise is finally shown to be suitable for the investigation of the eddy
dissipation produced by small scale turbulence.

\end{abstract}

\tableofcontents

\section{Introduction}

The theory of Stochastic Partial Differential Equations (SPDEs) is nowadays
very well developed, see for instance \cite{DaPZa92, Hyt vol 1, PrRo07, RozLot}, with many contributions on fluid dynamics
models, like \cite{Breit Fe Hof, Ch78, Fla08, Fla11, Kuk Shi 2012, ViFu88}. However, with the exception of the
literature making use of Kraichnan noise, which is motivated in Fluid Dynamics
by its invariance and scaling properties, in most cases there is no discussion
about the origin of noise and its form, in connection with the fact that it is
part of a fluid dynamic model. The purpose of this work is to introduce an
example of noise based on vortex structures, both in 2D (point vortices) and
3D (vortex filaments). We discuss its motivations and interest for the
understanding of fluid properties.

Some preliminary forms in 2D have been introduced in \cite{FGL, Grotto}, but the noise defined here is different and goes much beyond,
in particular because we treat the 3D case on the basis of the theory of
random vortex filaments, see Section \ref{subsec:3D-noise}.

Usually, in general or theoretical works on SPDEs, the noise is either
specified by means of its covariance operator, or by means of a finite or
countable sum of space-functions multiplied by independent Brownian motions.
Here we start from a different viewpoint. Motivated by the emergence of vortex
structures in turbulent fluids, we idealize their production/emergence process
by means of a sequence of vortex impulses, mathematically structured using a
jump process taking values in a set of vortex structures. This is described in
Section \ref{sec:jump-noise}. A suitable scaling limit of this jump process gives rise to
a Gaussian noise in a suitable Hilbert space. Different examples of such noise
depend on different choices of the vortex structures and their statistics, at
the level of the jump process. An heuristic picture then emerges, of a process
that fluctuates very rapidly between the elements of a family of vortex
structures. And the realizations of this noise are made of vortex structures
which idealize those observed in turbulent fluids - point vortices in 2D and
vortex filaments in 3D.

This noise is motivated by turbulent fluids. In the Physical literature, the
most common noise related to turbulence are the Fractional Gaussian Field
(FGF) and Kraichnan noise, see for instance \cite{Apolinario, Chaves, Eyink, Eyink-2, Krai, Krai-3, Kr, Lodhia}. In Section \ref{sec:covariance} we show that, on a torus in two and three dimensions,
the vortex noise covers FGF and Kraichnan noise by a special choice of the
statistical properties of the regularization parameter and the vortex
intensity. The vortex noise is thus a flexible ensemble - it may cover also
multifractal formalisms, see also \cite{FG} - and its realizations are the limit,
as described in Sections \ref{sec:jump-noise}-\ref{sec:noise}, of localized-in-space vortex structures similar
to those observed in turbulent fluids.

Finally, another main motivation for this investigation has been the recent
results on eddy dissipation, showing that a transport type noise depending in
a suitable way on a scaling parameter, in a transport-diffusion equation, in
the scaling limit gives rise to an additional diffusion operator
\cite{Galeati, FGL}. These results requires that the covariance
function of the noise, computed along the diagonal, $Q\left(  x,x\right)  $,
is large; but the operator norm of the covariance is small. We check when the
vortex noise satisfies these conditions. Heuristically speaking, they are
satisfied when, in the scaling limit, the vortex structures defining the noise
are more and more concentrated at \textit{small scales}. This confirms the
belief that eddy diffusion is a consequence of turbulence, but only when it is
suitably small scale.

\section{Jump noise and its Gaussian limit}
\label{sec:jump-noise}

\subsection{Why jump vortex noise in fluid modeling}

When a fluid moves through the small obstacles of a boundary (hills, trees,
houses for the lower surface wind, mountains for the lower atmospheric layer,
coast irregularities for the sea, vegetation for a river) or it moves through
small obstacles in the middle of the domain (like islands in the sea),
vortices are created by these obstacles, sometimes with a regular rhythm (von Kármán vortices) sometimes else more irregularly. In principle, these vortices
are the deterministic consequence of the dynamical interaction between fluid
and structure but in very many applications we never write the details of
those obstacles, when a larger scale investigation is done. Hence it is
reasonable to re-introduce the appearance of these vortices, so important for
turbulence, in the form of an external perturbation of the equations of motion.

Assume that the velocity field at time $t$ is $u\left(  t,x\right)  $. We may
idealize the modification of $u\left(  t,x\right)  $ due to the emergence of a
new vortex near an obstacle as an event occuring in a very short time around
time $t$, so that we have a jump:%
\[
u\left(  t^{+},x\right)  =u\left(  t^{-},x\right)  +\sigma\left(  x\right)
\]
where $\sigma\left(  x\right)  $ is presumably localized in space and
corresponds to a vortex structure. Continuum mechanics does not make jumps; we
idealize a fast change due to an instability as a jump, for a cleaner
mathematical description.

We may develop the previous idea in two directions. The simplest one is
suitable for investigations like the effect of turbulence on passive scalars
\cite{Chaves}, where a simple model of random velocity field is chosen: we consider a
stepwise constant velocity field with jumps like those described above; later
on we shall take a suitable scaling limit and get a Gaussian velocity field,
delta correlated in time, with space correlation of very flexible form. A more
elaborate proposal is to consider the Navier-Stokes equations with an
impulsive force given by a process with jumps:

\[
\partial_{t}u+u\cdot\nabla u+\nabla p=\nu \Delta u+\sum_{k\in K}\sum_{i}%
\delta\left(  t-t_{i}^{k}\right)  \sigma_{k}.
\]
Here $K$ is an index set and, for each $k\in K$, we denote by $t_{1}^{k}%
<t_{2}^{k}<...$ the sequence of jump times of class $k$ and by $\sigma_{k}$
the vortex structure (described at the level of velocity field) arisen at time
$t_{i}^{k}$. This way the fluid moves according to the free Navier-Stokes
equations between two consecutive jumps times. In the next section we
formalize the noise $\sum_{k\in K}\sum_{i}\delta\left(  t-t_{i}^{k}\right)
\sigma_{k}$ or more precisely, similarly to what it is done for White noise
and Brownian motion, we formalize the time integral of this distributional
process:%
\begin{equation}
W_{t}^{0}\left(  x\right)  =\sum_{k\in K}\sum_{i}1\left\{  t\geq t_{i}%
^{k}\right\}  \sigma_{k}. \label{PPP heuristic}%
\end{equation}
In this first heuristic formulation it is natural to introduce an index set
$K$ but below we shall avoid this.

\subsection{Jump vortex noise}

Given an open domain $\D\subset\mathbb{R}^{d}$, $d=2,3$, denote by
$C_{\text{c,sol}}^{\infty}\left(  \D,\mathbb{R}^{d}\right)  $ the space of
smooth solenoidal vector fields with compact support in $\D$, and denote by $H$
the closure of $C_{\text{c,sol}}^{\infty}\left(  \D,\mathbb{R}^{d}\right)  $ in
$L^{2}\left(  \D,\mathbb{R}^{d}\right)  $. One can prove, under some regularity
of the boundary, that $u\in H$ is an $L^{2}\left(  \D,\mathbb{R}^{d}\right)
$-vector field, with distributional divergence equal to zero, tangent to the
boundary \cite{Temam1}. The norm $\left\Vert u\right\Vert _{H}$ is given by
$\left\Vert u\right\Vert _{H}^{2}=\int_{\D}\left\vert u\left(  x\right)
\right\vert ^{2}dx$.

The following scheme is taken from M\'etivier \cite{Metivier}, first three
Chapters. The main tightness and convergence results for martingales, as
described in \cite{Metivier}, are due to Rebolledo \cite{Rebolledo}.

Let $P$ be a Borel probability measure on $H$. Assume that
\begin{equation}
\int_{H}\varphi\left(\left\Vert h\right\Vert _{H}\right)P(dh)<\infty\label{new-assump}%
\end{equation}
for some nondecreasing $\varphi:\mathbb{R}_{+}\rightarrow\mathbb{R}_{+}$ that
grows faster than quadratic, i.e.
\begin{equation}
\lim_{n\rightarrow\infty}\frac{\varphi(n)}{n^{2}}=\infty.\label{sup-quadr}%
\end{equation}
Denote by $Q_{P}$ the trace class covariance operator defined as
\[
Q_{P}=\int_{H}h\otimes hP\left(  dh\right)  .
\]
Assume $P$ has zero average
\begin{equation}
m_{P}=\int_{H}hP\left(  dh\right)  =0.\label{assumpt zero average}%
\end{equation}
We may also define, a.s. in $x,y\in \D$, the covariance (matrix-valued)
function%
\[
Q_{P}\left(  x,y\right)  =\int_{H}h\left(  x\right)  \otimes h\left(
y\right)  P\left(  dh\right)  .
\]
Indeed, $\int_{H}\left\vert h\left(  x\right)  \right\vert ^{2}P\left(
dh\right)  <\infty$ for a.e. $x\in \D$, thanks to Fubini-Tonelli theorem, since
$\int_{H}\left(  \int_{\D}\left\vert h\left(  x\right)  \right\vert
^{2}dx\right)  P\left(  dh\right)  <\infty$.

Consider the continuous time jump Markov process in $H$ with law of jumps
\[
p\left(  v,v+A\right)  =\frac{1}{\tau}P\left(  A\right)
\]
($v\in H$, $A\in\mathcal{B}\left(  H\right)  $), namely with infinitesimal
generator%
\[
\left(  \mathcal{L}F\right)  \left(  v\right)  =\frac{1}{\tau}\int_{H}\left(
F\left(  v+h\right)  -F\left(  v\right)  \right)  P\left(  dh\right)
\]
for all bounded continuous functions $F:H\rightarrow\mathbb{R}$. Here $\tau>0$
is the average interarrival between jumps. Denote by $W_{t}^{0}$ the
corresponding Markov process with initial condition $W_{0}^{0}=0$. Dynkin
formula%
\[
F\left(  W_{t}^{0}\right)  -F\left(  0\right)  =\int_{0}^{t}\left(
\mathcal{L}F\right)  \left(  W_{s}^{0}\right)  ds+M_{t}^{F}%
\]
gives us a the decomposition in a finite variation plus a martingale term.
Consider first the case when $F_{1}\left(  v\right)  =v$ (here we do not write
down classical details, namely that the computation should be done for a
continuous bounded cut-off of each component $\left\langle v,e_{i}%
\right\rangle $, where $\left(  e_{i}\right)  $ is a complete orthonormal
system, see \cite[page 14]{Metivier}). One has%
\[
\left(  \mathcal{L}F_{1}\right)  \left(  v\right)  =\frac{1}{\tau}\int%
_{H}\left(  v+h-v\right)  P\left(  dh\right)  =0
\]
because $m_{P}=0$. Hence $W_{t}^{0}=M_{t}^{F_{1}}$ namely the process
$W_{t}^{0}$ is a martingale. Let us compute its Hilbert-space-valued Meyer
process $\left\langle \left\langle W^{0}\right\rangle \right\rangle _{t}$. We
use the function $F_{2}\left(  v\right)  =v\otimes v$ (again one has to do the
computation first for a cut-off of the functions $\left\langle v,e_{i}%
\right\rangle \left\langle v,e_{j}\right\rangle $):%
\begin{align*}
\left(  \mathcal{L}F_{2}\right)  \left(  v\right)   &  =\frac{1}{\tau}\int%
_{H}\left(  \left(  v+h\right)  \otimes\left(  v+h\right)  -v\otimes v\right)
P\left(  dh\right) \\
&  =\frac{1}{\tau}\int_{H}\left(  v\otimes h+h\otimes v+h\otimes h\right)
P\left(  dh\right) \\
&  =\frac{1}{\tau}Q_{P}.
\end{align*}
Therefore $W_{t}^{0}\otimes W_{t}^{0}=\frac{t}{\tau}Q_{P}+M_{t}^{F_{2}}$. The
Meyer process $\left\langle \left\langle W^{0}\right\rangle \right\rangle
_{t}$ is thus (see the definition in \cite[pp. 8-12]{Metivier})
\[
\left\langle \left\langle W\right\rangle \right\rangle _{t}=\frac{t}{\tau
}Q_{P}.
\]

\subsection{Convergence of the rescaled process to a Brownian motion}

Let us now parametrize and rescale the previous process. We take average
interarrival between jumps given by
\[
\tau_{N}=\frac{1}{N^{2}}%
\]
and we reduce by $\frac{1}{N}$ the size of jumps by considering a probability
measure $P_{N}$ on $H$ with zero average $m_{P}=\int_{H}hP_{N}\left(
dh\right)  =0$ and covariance $Q_{P_{N}}$ given by%
\[
Q_{P_{N}}=\frac{1}{N^{2}}Q_{P}.
\]
Consider the associated process $W_{t}^{N}$, a martingale with Meyer process%
\[
\left\langle \left\langle W^{N}\right\rangle \right\rangle _{t}=\frac{t}%
{\tau_{N}}Q_{P_{N}}=tQ_{P}.
\]

\begin{definition}
Given $Q_{P}$, denote by $\left(  W_{t}\right)  _{t\geq0}$ a Brownian motion
on $H$ with incremental covariance $Q_{P}$.
\end{definition}

\begin{theorem}\label{thm:inv-prin}
The process $\left(  W_{t}^{N}\right)  _{t\geq0}$ converges in law to $\left(
W_{t}\right)  _{t\geq0}$, uniformly on every compact set of time, as processes
with values in $H$.
\end{theorem}

\begin{proof}
Using classical theorem of tightness for martingales (cf. \cite{Metivier},
Chapter 2, \cite{Rebolledo}), we have that the family of laws of the processes
$\left(  W^{N}\right)  _{N}$ is tight in the Skorohod space (because the
family of laws of $\left\langle \left\langle W^{N}\right\rangle \right\rangle
$ is tight) and every convergent subsequence has limit given by the law of a
martingale $W_{t}$ with $W_{0}=0$ and Meyer process%
\[
\left\langle \left\langle W\right\rangle \right\rangle _{t}=tQ_{P}.
\]
If we establish that $W$ has continuous paths, then it is a Brownian motion
with incremental covariance $Q_{P}$. One can prove that%
\begin{equation}
\lim_{N\rightarrow\infty}\mathbb{P}\left(  \sup_{s\in\left[  0,T\right]
}\left\Vert \Delta_{s}W^{N}\right\Vert _{H}>\epsilon\right)
=0\label{tail-incre}%
\end{equation}
where $\left\Vert \Delta_{s}W^{N}\right\Vert _{H}$ is the size of the jump (if
any) at time $s$ ($W^{N}$ is c\`adl\`ag). Since the set $\left\{  \sup
_{s\in\left[  0,T\right]  }\left\Vert \Delta_{s}w^{N}\right\Vert _{H}%
>\epsilon\right\}  $ is open in the Skorohod topology, from Portmanteau
theorem we get
\[
\mathbb{P}\left(  \sup_{s\in\left[  0,T\right]  }\left\Vert \Delta
_{s}W\right\Vert _{H}>\epsilon\right)  =0
\]
for every $\epsilon>0$, hence $W$ is continuous. To show (\ref{tail-incre}),
denote by $\{s_{i}\}_{i=0}^{N_{T}}\subset\lbrack0,T]$ the Poisson $(\tau
_{N}^{-1})$ arrival times, then we have that
\begin{align*}
\mathbb{P}\left(  \sup_{s\in\left[  0,T\right]  }\left\Vert \Delta_{s}%
W^{N}\right\Vert _{H}>\epsilon\right)   &  =1-\mathbb{P}\left(  \bigcap
_{\{s_{i}\}\subset\lbrack0,T]}\left\{  \left\Vert \Delta_{s_{i}}%
W^{N}\right\Vert _{H}\leq\epsilon\right\}  \right)  \\
&  =1-\mathbb{E}\left[  \prod_{\{s_{i}\}\subset\lbrack0,T]}\mathbb{P}\left(
\left\Vert \Delta_{s_{i}}W^{N}\right\Vert _{H}\leq\epsilon\;\left\vert
\;\{s_{i}\}_{i=0}^{N_{T}}\right.  \right)  \right]  \\
&  =1-\mathbb{E}\left[  \left[  1-\mathbb{P}\left(  \left\Vert \Delta
W^{N}\right\Vert _{H}>\epsilon\right)  \right]  ^{N_{T}}\right]
\end{align*}
where we used that given the Poisson arrival times, the laws of each jump size
$\left\Vert \Delta_{s_{i}}W^{N}\right\Vert _{H}$ is independent of it, and
identically distributed as what we simply denote by $\left\Vert \Delta
W^{N}\right\Vert _{H}$. By the elementary inequality $(1-y)^{n}\geq1-ny$ for
any $y\in\lbrack0,1]$ and $n\in\mathbb{N}$, and Markov's inequality, we have
that
\begin{align*}
\mathbb{P}\left(  \sup_{s\in\left[  0,T\right]  }\left\Vert \Delta_{s}%
W^{N}\right\Vert _{H}>\epsilon\right)   &  \leq\mathbb{E}[N_{T}]\mathbb{P}%
\left(  \left\Vert \Delta W^{N}\right\Vert _{H}>\epsilon\right)  \\
&  \leq\frac{TN^{2}}{\varphi(N\epsilon)}\mathbb{E}\left[  \varphi(\left\Vert
\Delta W\right\Vert _{H})\right]  \\
&  =\frac{TN^{2}}{\varphi(N\epsilon)}\int_{H}\varphi(\left\Vert h\right\Vert
_{H})P(dh)
\end{align*}
which is finite by (\ref{new-assump}), and converges to zero as $N\rightarrow
\infty$ by (\ref{sup-quadr}).
\end{proof}

\subsection{Reformulation as a PPP}

This is a side section which however may help the intuition (see also
(\ref{PPP heuristic})): we reformulate the jump process $W_{t}^{0}$ as a
Poisson Point Process (PPP). On a probability space $\left(  \Omega
,\mathcal{F},\mathbb{P}\right)  $, Let $\mathcal{P}$ be a PPP on
$[0,\infty)\times H$ with intensity measure $\lambda Leb\otimes P$, where
$\lambda Leb$ is Lebesgue measure scaled by $\lambda>0$, and $P$ is the
probability measure introduced in the previous subsections. Heuristically%
\[
\mathcal{P}\left(  dt,du\right)  =\sum_{i}\delta_{\left(  t_{i},\sigma
_{i}\right)  }\left(  dt,du\right)
\]
where $\left(  t_{i},\sigma_{i}\right)  $ is an i.i.d. sequence with $t_{i}$
``uniformly distributed on $[0,\infty)$", $\sigma_{i}$ distributed according to
$P$, $t_{i}$ and $\sigma_{i}$ independent of each other. Define the vector
valued random field, defined on $\left(  \Omega,\mathcal{F},\mathbb{P}\right)
$,
\[
W_{t}^{0}\left(  x\right)  =\sum_{t_{i}\leq t}\sigma_{i}\left(  x\right)
=\sum_{i}\sigma_{i}\left(  x\right)  1\left\{  t_{i}\leq t\right\}  .
\]
Compared to (\ref{PPP heuristic}), we may think that $K$ in that formula was a
finite set and we have simply reordered the jump times $\left(  t_{i}%
^{k}\right)  $ in a single sequence $\left(  t_{i}\right)  $ and we have
renamed the jump velocity fields. This definition is slightly heuristic
because it makes use of the representation as infinite sum which is true only
in a suitable limit sense; a rigorous definition of $W\left(  t,x\right)  $ is%
\[
W_{t}^{0}\left(  x\right)  =\int_{[0,\infty)\times H}u\left(  x\right)
1\left\{  t^{\prime}\leq t\right\}  \mathcal{P}\left(  dt^{\prime},du\right)
.
\]
However, in the sequel, for the sake of interpretability, we shall always use
the heuristic expressions.

The intuition is that eddies $\sigma_{i}\left(  x\right)  $ are chosen at
random with distribution $P$, with exponential inter-arrival times of rate
$\lambda$. Condition (\ref{assumpt zero average}) asks, heuristically
speaking, that both an eddy and its opposite are equally likely to be chosen.

Rescale $W_{t}^{0}\left(  x\right)  $ as%
\[
W_{t}^{N}\left(  x\right)  =\frac{1}{N}\sum_{i}\sigma_{i}\left(  x\right)
1\left\{  t_{i}\leq N^{2}t\right\}  .
\]
Let us compute the expectation and the covariance function of this process.
One has ($\mathbb{E}$ denotes the Mathematical expectation on $\left(
\Omega,\mathcal{F},\mathbb{P}\right)  $)
\[
\mathbb{E}\left[  W_{t}^{N}\left(  x\right)  \right]  =0
\]
from the independences and condition (\ref{assumpt zero average}). Moreover,%
\[
\mathbb{E}\left[  W_{t}^{N}\left(  x\right)  \otimes W_{t}^{N}\left(
y\right)  \right]  =\frac{1}{N^{2}}\sum_{i}\mathbb{E}\left[  \sigma_{i}\left(
x\right)  \otimes\sigma_{i}\left(  y\right)  1\left\{  t_{i}\leq
N^{2}t\right\}  \right]
\]
having used the independence when $i\neq j$ and property
(\ref{assumpt zero average}) again; hence%
\[
=\frac{Q_{P}\left(  x,y\right)  }{N^{2}}\sum_{i}\mathbb{P}\left(  t_{i}\leq
N^{2}t\right)  .
\]

\begin{proposition}%
\[
\sum_{i}\mathbb{P}\left(  t_{i}\leq N^{2}t\right)  =N^{2}\lambda t.
\]
Hence%
\[
\mathbb{E}\left[  W_{t}^{N}\left(  x\right)  \otimes W_{t}^{N}\left(
y\right)  \right]  =\lambda tQ_{P}\left(  x,y\right)  .
\]

\end{proposition}

\begin{proof}
We note that
\[
\sum_{i}\mathbb{P}\left(  t_{i}\leq N^{2}t\right)  =\mathbb{E}\left[  \sum
_{i}1\left\{  t_{i}\leq N^{2}t\right\}  \right]  =\mathbb{E}\left[
\eta_{\lambda}(N^{2}t)\right]  =N^{2}\lambda t.
\]
where $\eta_{\lambda}(\cdot)$ denotes a Poisson process on $\mathbb{R}_{+}$
with intensity $\lambda$.
\end{proof}

This is another way of seeing the link between the noise with jumps and the
covariance of the limit Brownian motion.

\section{Examples in 2D and 3D}\label{sec:noise}

The mathematical object discussed in the previous section, although initially
motivated by vortex structures, were completely general: given any probability measure $P$ on $H$ with covariance $Q_P$, the previous
construction and results apply and defines a Brownian motion $W_t$ in $H$ with covariance operator $Q_P$. Notice that $P$ is not necessarily Gaussian: $P$ and $W_{1}$ both have covariance $Q_P$, but only $W_{1}$ needs to be Gaussian. In a sense,
we ``realize" approximately samples of the Brownian motion $W_{t}$ by means of
samples of a possibly "nonlinear" (non Gaussian) process $W_{t}^{N}$.

In this section we give our two main examples of the measure $P$, highly non
Gaussian. It is inspired by vortex structures.

Common to both descriptions are a few objects. First, given $\delta>0$, we
define
\[
\D_{\delta}:=\left\{  x\in \D:\;\text{dist}(x,\D^{c})>\delta\right\}  .
\]
Second, we have a filtered probability space $\left(  \Omega,\mathcal{F}%
,\mathcal{F}_{t},\mathbb{P}\right)  $ and several $\mathcal{F}_{0}$-measurable
r.v.'s: a) $X_{0}$ with law $p_{0}\left(  dx\right)  $ supported on
$\D_{\delta}$, which will play the role of the center of the vortex in 2D and
the initial position of the vortex filament in 3D; b) $\Gamma $, real
valued, with the physical meaning of circulation, with%
\begin{align*}
\mathbb{E}\left[  \Gamma\right]    & =0,\qquad\mathbb{E}\left[  \left\vert
\Gamma\right\vert ^{p}\right]  <\infty\text{ for some }p>2\\
\sigma^{2}  & :=\mathbb{E}\left[  \Gamma^{2}\right]  ;
\end{align*}
c) $L$, positive valued, randomizing the size of the mollification, with the
property%
\begin{align}\label{small-L}
\mathbb{P}\left(  L\in\left(  0,\delta/2\right)  \right)  =1;
\end{align}
d) $U$, positive valued, randomizing the length of the vortex filament.
Moreover, in 3D, we also have: e) a Brownian motion on $\left(  \Omega
,\mathcal{F},\mathcal{F}_{t},\mathbb{P}\right)  $ with values in
$\mathbb{R}^{3}$. In the 2D case we just take $\mathcal{F}=\mathcal{F}_{0}$
and do not need the filtration. 

For sake of simplicity of exposition, we shall assume that $X_{0}%
,\Gamma, L$ and $U$ are independent, but most of the results can be extended to more general cases. 

The last common element of the theory is a smooth symmetric probability
density $\theta$ supported in the ball $B\left(  0,1\right)  $ and its
rescaled mollifiers 
\begin{align}\label{dirac}
\theta_{\ell}\left(  x\right)  =\ell^{-d}\theta\left(
\ell^{-1}x\right), 
\end{align}
with support in $B\left(  0,\ell\right)  $. In particular, $\theta_\ell\ge 0$ and $\int\theta_\ell(x)dx=1$.

\subsection{Point vortices and definition of $P$ in the 2D case}

In 2D, by a point vortex we mean a vorticity field of delta Dirac type,
$\delta_{x_{0}}$; its use in 2D fluid mechanics is manifold, see for instance
\cite{MarchioroPulvirenti}. If the vorticity is assumed distributional and
equal to $\delta_{x_{0}}$, with $x_{0}$ in the interior of $\D$, then the
so-called stream function $\psi_{\D,x_{0}}$ is given by the solution of%
\begin{align*}
-\Delta\psi_{\D,x_{0}} &  =\delta_{x_{0}}\qquad\text{in }\D\\
\psi_{\D,x_{0}}|_{\partial \D} &  =0
\end{align*}
and the associated velocity vector field is given by:%
\[
u_{\D,x_{0}}\left(  x\right)  =\nabla^{\perp}\psi_{\D,x_{0}}\left(  x\right)
\]
where $\nabla^{\perp}f=\left(  \partial_{2}f,-\partial_{1}f\right)  $. One has%
\[
\psi_{\D,x_{0}}\left(  x\right)  =\frac{1}{2\pi}\log\frac{1}{\left\vert
x-x_{0}\right\vert }+h_{\D,x_{0}}\left(  x\right)
\]
where $h_{\D,x_{0}}$ is a smooth function, solution of the problem%
\begin{align*}
-\Delta_{x}h_{\D,x_{0}} &  =0\text{ in }\D\\
h_{\D,x_{0}}\left(  x\right)   &  =\frac{1}{2\pi}\log\left\vert x-x_{0}%
\right\vert \text{ for }x\in\partial \D.
\end{align*}
In the sequel, as it is customary, we shall denote $u_{\D,x_{0}}\left(
x\right)  $ simply by $K\left(  x,x_{0}\right)  $. Hence%
\begin{align}\label{2d-kernel}
K\left(  x,x_{0}\right)  =-\frac{1}{2\pi}\frac{\left(  x-x_{0}\right)  ^{\perp
}}{\left\vert x-x_{0}\right\vert ^{2}}+\nabla^{\perp}h_{\D,x_{0}}\left(
x\right)  ,
\end{align}
where $x^\perp=(x_2, -x_1)$.

Recall that $\theta_\ell$ \eqref{dirac}, as $\ell\to0$ is an approximation of the Dirac delta function. Expressions of the form $\theta_{\ell}(x-x_{0})$ are idealized smoothed point
vortices, at the vorticity level, and the associated velocity field is
\begin{align*}
K_\ell(x,x_0):=\int_{D}K\left(  x,y\right)  \theta_{\ell}\left(  y-x_{0}\right)  dy.
\end{align*}
With
these preliminaries, let us define $P$.

\begin{definition}
\label{def P 2D}In the 2D case, the probability measure $P$ on the space $H$
is the law of the $H$-valued r.v.%
\begin{align}\label{def-2d}
\Gamma K_L(x,X_0)=\Gamma\int_{\D}K\left(  x,y\right)  \theta_{L}\left(  y-X_{0}\right)  dy.
\end{align}
\end{definition}

For future reference, the spatial covariance matrix of the vortex noise  in 2D is given by 
\begin{align}\label{smoothing-1}
Q_{\text{vortex}}(x,x')=\E\left[\Gamma K_L(x,X_0) \otimes \Gamma K_L(x',X_0) \right], \quad x,x'\in\D.
\end{align}

\begin{proposition}
The random vector field of Definition \ref{def P 2D} takes values in $H$. If%
\[
\E\left(  |\Gamma|^{p}L^{-p}\right)  <\infty,
\]
for some $p>2$, then it satisfies (\ref{new-assump})-(\ref{sup-quadr}).
Moreover, it satisfies (\ref{assumpt zero average}).
\end{proposition}

\begin{proof}
Fixing any $p>2$, by H\"older's inequality,
\begin{align*}
\int_H \|h\|_H^pP(dh)&=\E\left[\left(\int_{\D}|\Gamma K_{L}(x,X_{0})|^2dx\right)^{p/2}\right]\\
&=\E\left(|\Gamma|^2\int_{\D}\left|\int_\D K(x,y)\theta_L(y-X_0)dy\right|^2dx\right)^{p/2}\\
&\le \E\left(|\Gamma|^p|\D|^{\frac{p}{2}-1}\int_{\D}\left|\int_\D K(x,y)\theta_L(y-X_0)dy\right|^pdx\right),
\end{align*}
where recall that 
\[
K(x,y)=\frac{1}{2\pi}\frac{(x-y)^\perp}{|x-y|^2}+h_{\D,y}(x).
\]
Since $X_0\in \D_{\delta}$ and $|y-X_0|\le L<\delta/2$ \eqref{small-L}, for all $y$ contributing to the above integral, we have $y\in \D_{\delta/2}$. Therefore, the part $\nabla_x^\perp h_{\D,y}(x)$ of the kernel $K(x,y)$ is smooth as a function of $x\in \D$, for every $y\in \D_{\delta/2}$. Due to continuous dependence of $h_{\D,y}(x)$ on boundary conditions, hence on the variable $y$, the following constant is finite
\[
C(\D,\delta):=\sup_{y\in \ovl \D_{\delta/2}}\sup_{x\in \ovl\D}|\nabla_x^\perp h_{\D,y}(x)|.
\]
The contribution of $\nabla_x^\perp h_{\D,y}(x)$ to the above integral is a.s. finite:
\begin{align*}
&\left|\int_\D\nabla_x^\perp h_{\D,y}(x)\theta_L(y-X_0)dy\right|^p\\
&\le \left(\int_\D\left|\nabla_x^\perp h_{\D,y}(x)\right|\theta_L(y-X_0)dy\right)^p\\
&\le C(\D,\delta)^p\left(\int_\D\theta_L(y-X_0)dy\right)^p\le  C(\D,\delta)^p,
\end{align*}
where we used that $\int_\D\theta_L(y-X_0)dy =1$ for every realization of $X_0$.  It suffices now to focus on the other part of the kernel $(2\pi)^{-1}\frac{\left(  x-y\right)  ^{\perp}%
}{\left\vert x-y\right\vert ^{2}}$. We have that 
\begin{align*}
&\left|\int_\D\frac{\left(  x-y\right)  ^{\perp}%
}{\left\vert x-y\right\vert ^{2}}\theta_L(y-X_0)dy\right|^p\\
&\le\left(\int_\D\frac{1}{|x-y|}\theta_L(y-X_0)dy\right)^p\\
&\overset{y'=L^{-1}y}{=}  \left(\int_{L^{-1}\D}\frac{L^{-1}}{|L^{-1}x- y'|}\theta\left(y'-L^{-1}X_0\right)dy'\right)^p\\
&\le \|\theta\|_{L^\infty}^p\left(L^{-1}\int_{B(L^{-1}X_0,1)}\frac{1}{|L^{-1}x- y'|}dy'\right)^p\\
&\le \|\theta\|_{L^\infty}^p\left(L^{-1}\int_{B(L^{-1}(x-X_0),1)}\frac{1}{| y''|}dy''\right)^p\\
&\le \|\theta\|_{L^\infty}^p\left(L^{-1}\int_{B(0,1)}\frac{1}{| y''|}dy''\right)^p\\&
\le C_{p,\theta}L^{-p},
\end{align*}
where we use the fact that the integral of $|y''|^{-1}$ over a unit ball centered anywhere in $\R^2$ is maximized when the center is the origin, and the constant $C_{p,\theta}$ is nonrandom and independent of $x$. Hence, we get that 
\begin{align*}
\int_H\|h\|_H^pP(dh)\le C_{p,\theta}|\D|^{\frac{p}{2}}\E\left(|\Gamma|^pL^{-p}\right).
\end{align*}

Finally, it satisfies (\ref{assumpt zero average}):%
\[
\mathbb{E}\left[  \Gamma\int_{\D}K\left(  \cdot,y\right)  \theta_{L}\left(
y-X_{0}\right)  dy\right]  =\mathbb{E}\left[  \Gamma\right]  \mathbb{E}\left[
\int_{\D}K\left(  \cdot,y\right)  \theta_{L}\left(  y-X_{0}\right)  dy\right]
=0
\]
because the second expectation is finite and the first one is equal to zero,
by assumption.
\end{proof}

The case when $L=0$ is outside the previous definition and result. The
velocity field $K\left(  x,x_{0}\right)  $ is not of class $H$. Nevertheless
it is of class $L^{p}\left(  \D,\mathbb{R}^{2}\right)  $ for $p<2$, or of class
$H^{-s}\left(  \D,\mathbb{R}^{2}\right)  $ for $s>0$. Therefore we may consider
the random field%
\[
\Gamma K\left(  x,X_{0}\right)
\]
taking values in these spaces and call $P$ its law. We shall see below that it
satisfies certain special properties.

\subsection{Vortex filaments and the definition of $P$ in the 3D case}
\label{subsec:3D-noise}

In 3D, by vortex filament we mean a distributional vector valued field (a
``current", in the language of Calculus of Variations \cite{Giaquinta}), given
by
\[
\int_{0}^{U\wedge\tau}\delta_{X_{t}}dX_{t}%
\]
where $X_{t}$ is a function or a process such that the previous expression is
well defined. We have already introduced a possibly relevant stopping time
$\tau$ because it may help to cope with the presence of a boundary. Stochastic
currents have been introduced and investigated in some works
\cite{FlaGubGiaTor, FlaGubRusso, Capasso, BessaihCoghi}. We do not need,
strictly speaking, that theory here since we shall always deal with mollified
objects. 
In this
work we shall always assume that $\left(  X_{t}\right)  $ has the law of a
Brownian motion, but it is interesting to investigate also other processes,
for instance directed polymers, like in \cite{Lions-Majda}. 

The following construction of a vortex filament in 3D is due to \cite{FG}
(which we slightly modify). Let $(\Gamma, U, \ell)\in\mathbb{R}_{+}^{3}$ be a triple
whose joint distribution is given by some probability measure $\nu (d\gamma,
du, d\ell)$ (assumed to be a product measure for simplicity). Let $(X_{t})_{t\ge0}$ denote a 3D Brownian motion of independent components, starting from $X_0$ distributed with a probability
density $p_{0}(x)$ supported in $\D_{\delta}$, where $p_{0}%
(x)\in[p_{\text{min}},p_{\text{max}}]\subset(0,\infty)$. We call $\mathcal{W}$
its law which we assume to be independent of $\nu (\cdot, \cdot, \cdot)$. Define the first
exit time from $\D_{\delta}$ of $(X_{t})$ by
\begin{align}\label{stopping}
\tau=\tau^{\D_{\delta}}:=\inf\left\{t\ge0:\; X_{t}\in \D_{\delta}^{c}\right\}\in
[0,\infty).
\end{align}
We consider random vorticity fields defined as%
\[
\int_{0}^{U\wedge\tau}\left(  \theta\ast\delta_{X_{t}}\right)  \left(
x\right)  dX_{t}=\int_{0}^{U\wedge\tau}\theta(x-X_{t})dX_{t}.
\]
Let $A\left(  x\right)  $ be the vector potential defined path by path by the
solution of the equation%
\begin{align*}
-\Delta A\left(  x\right)   &  =\int_{0}^{U\wedge\tau}\theta(x-X_{t}%
)dX_{t} \quad \text{ in }\D\\
A|_{\partial \D} &  =0
\end{align*}
and extend $A=0$ outside of $\D$, when necessary. Then the associated velocity is
given by:%
\[
u\left(  x\right)  =\operatorname{curl}A\left(  x\right)  .
\]

Concerning Biot-Savart kernel, here we have%

\[
\psi_{\D,x_{0}}\left(  x\right)  =\frac{1}{4\pi}\frac{1}{\left\vert
x-x_{0}\right\vert }+h_{\D,x_{0}}\left(  x\right)
\]
where $h_{\D,x_{0}}$ is a smooth function, solution of the problem%
\begin{align*}
-\Delta_{x}h_{\D,x_{0}} &  =0\text{ in }\D\\
h_{\D,x_{0}}\left(  x\right)   &  =-\frac{1}{4\pi}\frac{1}{\left\vert
x-x_{0}\right\vert }\text{ for }x\in\partial \D.
\end{align*}
As usual, we shall denote $\operatorname{curl}\psi_{\D,x_{0}}\left(  x\right)
$ simply by $K\left(  x,x_{0}\right)  $, which now is vector valued and its
action on a generic vector $v$ is given by%
\begin{align}\label{kernel-3d}
K\left(  x,x_{0}\right)\times v  :=-\frac{1}{4\pi}\frac{\left(  x-x_{0}\right)  \times
v}{\left\vert x-x_{0}\right\vert ^{3}}+
\nabla_xh_{\D,x_{0}}\left(
x\right)  \times v.
\end{align}

\begin{definition}
\label{def P 3D}In the 3D case, the probability measure $P$ on the space $H$
is the law of the $H$-valued r.v.%
\begin{align}\label{def-3d}
\Gamma K_L(x, X_\cdot):=\Gamma\int_{\D}K\left(  x,y\right)\times  \left(  \int_{0}^{U\wedge\tau}\theta
_{L}\left(y-X_{t}\right)dX_{t}\right)  dy.
\end{align}
\end{definition}

\begin{remark}
We use the killed BM, not the normally reflected BM, in the definition of the
filament, because the latter is not a local martingale, only a semimartingale
due to the boundary push term, which leads to difficulties in integration
against $dX_{t}$.
\end{remark}

For future reference, the spatial covariance matrix of the vortex noise in 3D is given by 
\begin{align}\label{smoothing-2}
Q_{\text{vortex}}(x,x')=\E\left[\Gamma K_L(x,X_\cdot) \otimes \Gamma K_L(x',X_\cdot) \right], \quad x,x'\in\D.
\end{align}

\begin{proposition}
The random vector field of Definition \ref{def P 2D} takes values in $H$. If%
\[
\E\left(  |\Gamma |^{p}U^{\frac{p}{2}}L^{-2p}\right)  <\infty,
\]
for some $p>2$, then it satisfies (\ref{new-assump})-(\ref{sup-quadr}).
Moreover, it satisfies (\ref{assumpt zero average}).
\end{proposition}

\begin{proof}
Fix any $p>2$, we compute
\begin{align*}
\int_{H} \|h\|_H^{p}P(dh)&=\mathbb{E}\left[  \left(  \int_{\D}|\Gamma %
u(x)|^{2}dx\right)  ^{p/2}\right].
\end{align*}
Fixing any realization of $(\Gamma , U,L)$ according to measure $\nu $, we take expectation with respect to the Wiener measure $\mathcal{W}$ first. We compute, by  H\"older's inequality and $p/2>1$, and Burkholder–Davis–Gundy inequality
\begin{align*}
&\mathcal{W}\left[\left(  \int_{\D}|u(x)|^{2}dx\right)  ^{p/2}\right]\\
&=  \mathcal{W}\left[  \left(  \int_{\D}\left|  \int_{0}^{U\wedge \tau}\int_{\D_\delta}%
K(x,y)\theta_{L}\left(y-X_{t}\right)dy\times dX_{t}\right|  ^{2}dx\right)
^{p/2}\right] \\
&  \le|\D|^{\frac{p}{2}-1}\mathcal{W}\left( \int_{\D}\left|  \int_{0}^{U\wedge \tau}\int%
_{\D_\delta}K(x,y)\theta_{L}\left(y-X_{t}\right)dy\times dX_{t}\right|  ^{p}dx\right) \\
&  =|\D|^{\frac{p}{2}-1}\int_{\D} \mathcal{W}\left(  \left|  \int_{0}^{U\wedge \tau}\int%
_{\D_\delta}K(x,y)\theta_{L}\left(y-X_{t}\right)dy\times dX_{t}\right|  ^{p}\right)dx\\
&  \le |\D|^{\frac{p}{2}-1}\int_{\D}\mathcal{W}\left(  \left|  \int_{0}^{U\wedge \tau} 2\left|
\int_{\D_\delta}K(x,y)\theta_{L}\left(y-X_{t}\right)dy\right|  ^{2}dt\right|  ^{p/2}\right) dx .
\end{align*}
Since $X_{t\wedge\tau}\in \D_{\delta}$ \eqref{stopping}, we have that any $y$ that contributes to the above integral is supported in $y\in \D_{\delta/2}$, hence $\nabla_{x}h_{\D,y}(x)$ part of the kernel $K(x,y)$ is uniformly bounded, i.e. 
\[
\sup_{y\in \ovl \D_{\delta/2}}\sup_{x\in \ovl \D}|\nabla_{x}h_{\D,y}(x)|\le C(\D, \delta).
\]
Hence its contribution in the above integral can be computed, as for any $x\in \D$,
\begin{align*}
&\mathcal{W}\left(  \left|  \int_{0}^{U\wedge \tau}\left|
\int_{\D}\nabla_xh_{\D,y}(x)\theta_{L}(y-X_{t})dy\right|  ^{2}dt\right|  ^{p/2}\right)\\
&\le U^{\frac{p}{2}-1}\int_{0}^{U} \mathcal{W}\left(  \left|  \int_{\D}\nabla_xh_{\D,y}(x)\theta_{L}(y-X_{t})dy\right|  ^{p}1_{\{t\le\tau\}}\right) dt\\
&\le U^{\frac{p}{2}-1}C(\D,\delta)^p\int_{0}^{U} \mathcal{W}\left( \left(  \int_{\D}\theta_{L}(y-X_{t})dy\right)^{p}1_{\{t\le \tau\}}\right)dt \\
&\le U^{\frac{p}{2}}C(\D,\delta)^p
\end{align*}
using that $\int \theta_L(y-X_t)dy=1$ for every possible realization of $X_{t\wedge\tau}\in \D_{\delta}$.

It suffices to focus on the other part of the kernel $(4\pi)^{-1}\frac{x-y}{|x-y|^{3}}$. We can do an explicit calculation: by H\"older's inequality and then a change of variables, we have that for any $x\in \D$,

\begin{align*}
&  \mathcal{W}\left( \left|  \int_{0}^{U\wedge \tau}\left|  \int_{\D}\frac{x-y}%
{|x-y|^{3}}\theta_{L}(y-X_{t})dy\right|  ^{2}dt\right|  ^{p/2}\right)\\
&\le U^{\frac{p}{2}-1}\int_{0}^{U} \mathcal{W}\left(  \left|  \int_{\D}\frac{x-y}{|x-y|^{3}%
}\theta_{L}(y-X_{t})dy\right|  ^{p}1_{\{t\le \tau\}}\right)dt \\
& \overset{y^{\prime}=L^{-1}y}{=} \quad   U^{\frac{p}{2}-1}\int_{0}^{U} \mathcal{W}\left(
\left|  \int_{L^{-1}\D}\frac{L^{-2}(L^{-1}x-y^{\prime})}{|L^{-1}x-y^{\prime}|^3%
}\theta\left(y'-L^{-1}X_{t}\right)dy^{\prime}\right|  ^{p}1_{\{t\le \tau\}}\right)dt \\
&  \le U^{\frac{p}{2}-1}L^{-2p} \left\|\theta\right\|_{L^\infty}^p\int_{0}^{U} \mathcal{W}\left(
\left|  \int_{B(L^{-1}X_{t},1)}\frac{1}{|L^{-1}x-y^{\prime}|^2}dy^{\prime
}\right|  ^{p}1_{\{t\le \tau\}}\right)dt\\
& =U^{\frac{p}{2}-1}L^{-2p} \left\|\theta\right\|_{L^\infty}^p\int_{0}^{U} \mathcal{W}\left(
\left|  \int_{B\left(L^{-1}(x-X_{t}),1\right)}\frac{1}{|y^{''}|^2}dy^{''}\right|  ^{p}\right)dt \\
&\le U^{\frac{p}{2}}L^{-2p} \left\|\theta\right\|_{L^\infty}^p\; \mathcal{W}\left(
\left|  \int_{B\left(0,1\right)}\frac{1}{|y^{''}|^2}dy^{''}\right|  ^{p}\right) \\
& \le C_{p, \theta}U^{\frac{p}{2}}L^{-2p},
\end{align*}
where $C_{p,\theta}$ is a non-random constant independent of $x$. Indeed, we used the geometric fact that the integral of the function $|y^{''}|^{-2}$ over a unit ball centered at anywhere in $\R^3$, is maximized when the center is the origin. 

Thus, we can conclude that
\begin{align*}
\mathbb{E}\left[  \left(  \int_{\D}|\Gamma u(x)|^{2}dx\right)
^{p/2}\right]   &  \le C_{p,\theta}|\D|^{\frac{p}{2}}\E \left( |\Gamma |^{p}U^{\frac{p}{2}} L^{-2p} \right)  
\end{align*}
with the finiteness of the RHS providing a sufficient condition.

Finally, it satisfies (\ref{assumpt zero average}):%
\begin{align*}
& \mathbb{E}\left[  \Gamma\int_{\D}K\left(  \cdot,y\right) \times \left(  \int%
_{0}^{U\wedge\tau}\theta_{L}(y-X_{t})dX_{t}\right)  dy\right]  \\
& =\mathbb{E}\left[  \Gamma\right]  \mathbb{E}\left[  \int_{\D}K\left(
\cdot,y\right)  \times\left(  \int_{0}^{U\wedge\tau}\theta_{L}(y-X_{t}%
)dX_{t}\right)  dy\right]  =0
\end{align*}
because the second expectation is finite and the first one is equal to zero,
by assumption.
\end{proof}

\section{Vortex noises reproduce Fractional Gaussian Fields and Kraichnan noise}\label{sec:covariance}
In this section, we analyse the covariance operators of our vortex noises constructed above in 2D and 3D, and show that our vortex noises are instances of Fractional Gaussian Fields (\cite{Lodhia}, which is a broad class of Gaussian generalized random fields that includes Gaussian Free Field and Kraichnan noise). We show that, by choosing the statistical parameters of our model suitably, we can reproduce a large class of FGF. It may also reproduce multifractal vector fields, which was the main motivation of study in  \cite{FG}.

For simplicity, our fields are defined on the torus $\T^d$, $d=2,3$. 

In the scalar case and on the torus $\T^d=\mathbb{R}^{d}/\mathbb{Z}^{d}$, the
classical $d$-dimensional FGF of index $s\in\mathbb{R}$ is the Gaussian field
with covariance $\left(  -\Delta\right)  ^{-s}$, where $\Delta$ is the
Laplacian in on $\mathbb{T}^{d}$ (see \cite{Lodhia}). The case $s=1$ is called Gaussian
Free Field (GFF). Similarly let us introduce a Gaussian measure on solenoidal
vector fields. Let $H$ be the space of mean zero periodic $L^{2}$ solenoidal
vector fields. The Stokes operator is defined as%
\begin{align*}
A  &  :D\left(  A\right)  \subset H\rightarrow H\\
D\left(  A\right)   &  =\mathsf H^{2}\left(  \mathbb{T}^{d},\mathbb{R}^{d}\right) \\
Av  &  =\Delta v
\end{align*}
(no projection of $L^{2}\left(  \mathbb{T}^{d},\mathbb{R}^{d}\right)  $ to $H$
is needed here, opposite to the case of a bounded domain with Dirichlet
boundary conditions). The Laplacian $\Delta v$ is computed componentwise. The
operator $A$ is invertible in $H$ (see \cite{Temam1}). With these definitions
at hand, we call Solenoidal Fractional Gaussian Field (SFGF) of index
$s\in\mathbb{R}$ the Gaussian measure with covariance $\left(  -A\right)
^{-s}$. The case $s=1$ will be called Solenoidal Gaussian Free Field (SGFF).

\subsection{Covariance of 2D vortex noise}
Let us first consider the 2D case, and recall the definition of the noise
based on point vortices \eqref{def-2d}.

The covariance operator of our noise is given by%
\[
\left\langle \mathbb{Q}v,w\right\rangle =\mathbb{E}\left[  \Gamma^{2}%
\int_{\mathbb{T}^{2}}K_{L}\left(  x,X_{0}\right)  \cdot v\left(  x\right)
dx\int_{\mathbb{T}^{2}}K_{L}\left(  x^{\prime},X_{0}\right)  \cdot w\left(
x^{\prime}\right)  dx^{\prime}\right]  .
\]
Call $Q_{\text{vortex}}\left(  x,x^{\prime}\right)  $ its covariance function
(matrix-valued), such that
\[
\left\langle \mathbb{Q}v,w\right\rangle =\int_{\mathbb{T}^{2}}\int%
_{\mathbb{T}^{2}}v\left(  x\right)  ^{T}Q_{\text{vortex}}\left(  x,x^{\prime
}\right)  w\left(  x^{\prime}\right)  dxdx^{\prime}.
\]
It is clear (and proved below) that it is homogeneous:%
\[
Q_{\text{vortex}}\left(  x,x^{\prime}\right)  =Q_{\text{vortex}}\left(
x-x^{\prime}\right)
\]
for a matrix function $Q_{\text{vortex}}\left(  x\right)  $. In the sequel we
denote by $\mathbb{Z}_{0}^{d}$ the set $\mathbb{Z}^{d}\backslash\left\{
0\right\}  $.

\begin{proposition}
Assume $\theta$ symmetric, and $X_{0}$ independent of $\left(  \Gamma
,L\right)  $ and uniformly distributed. Then%
\begin{equation}
Q_{\text{vortex}}\left(  x\right)  =\sum_{\mathbf{k\in}\mathbb{Z}_{0}^{2}%
}\mathbb{E}\left[  \Gamma^{2}\left\vert \widehat{\theta}\left(  L\mathbf{k}%
\right)  \right\vert ^{2}\right]  \frac{1}{\left\vert \mathbf{k}\right\vert
^{2}}P_{\mathbf{k}}e^{i\mathbf{k}\cdot x}.\label{cov-2d}%
\end{equation}

\end{proposition}

\begin{proof}
We may rewrite
\begin{align*}
\int_{\mathbb{T}^{2}}K_{L}\left(  x,X_{0}\right)  \cdot v\left(  x\right)  dx
&  =\int_{\mathbb{T}^{2}}\int_{\mathbb{T}^{2}}K\left(  x,y\right)  \cdot
v\left(  x\right)  \theta_{L}\left(  y-X_{0}\right)  dydx\\
&  =\left(  \theta_{L}\ast K\ast v\right)  \left(  X_{0}\right)
\end{align*}
Therefore%
\begin{align*}
\left\langle \mathbb{Q}v,w\right\rangle  &  =\mathbb{E}\left[  \Gamma
^{2}\left(  \theta_{L}\ast K\ast v\right)  \left(  X_{0}\right)  \left(
\theta_{L}\ast K\ast w\right)  \left(  X_{0}\right)  \right]  \\
&  =\mathbb{E}\left[  \Gamma^{2}\int_{\mathbb{T}^{2}}\left(  \theta_{L}\ast
K\ast v\right)  \left(  x\right)  \left(  \theta_{L}\ast K\ast w\right)
\left(  x\right)  dx\right]  .
\end{align*}
By Parseval theorem%
\begin{align*}
\left\langle \mathbb{Q}v,w\right\rangle  &  =\mathbb{E}\left[  \Gamma^{2}%
\sum_{\mathbf{k}}\widehat{\theta_{L}\ast K\ast v}\left(  \mathbf{k}\right)
\overline{\widehat{\theta_{L}\ast K\ast w}\left(  \mathbf{k}\right)  }\right]
\\
&  =\sum_{\mathbf{k\in}\mathbb{Z}_{0}^{2}}\mathbb{E}\left[  \Gamma
^{2}\left\vert \widehat{\theta_{L}^{T}}\left(  \mathbf{k}\right)  \right\vert
^{2}\right]  \frac{1}{\left\vert \mathbf{k}\right\vert ^{2}}\left\langle
P_{\mathbf{k}}\widehat{v}\left(  \mathbf{k}\right)  ,\overline{\widehat{w}%
\left(  \mathbf{k}\right)  }\right\rangle
\end{align*}
recalling that
\[
\widehat{K}\left(  \mathbf{k}\right)  =i\frac{\mathbf{k}^{\perp}}{\left\vert
\mathbf{k}\right\vert ^{2}}%
\]
and calling $P_{\mathbf{k}}=I-\frac{\mathbf{k}\otimes\mathbf{k}}%
{|\mathbf{k}|^{2}}$ is the projection on the orthogonal to $\mathbf{k}$.
Therefore%
\[
Q_{\text{vortex}}\left(  x\right)  =\sum_{\mathbf{k\in}\mathbb{Z}_{0}^{2}%
}\mathbb{E}\left[  \Gamma^{2}\left\vert \widehat{\theta_{L}}\left(
\mathbf{k}\right)  \right\vert ^{2}\right]  \frac{1}{\left\vert \mathbf{k}%
\right\vert ^{2}}P_{\mathbf{k}}e^{i\mathbf{k}\cdot x}.
\]
Since $\widehat{\theta_{\ell}}\left(  \mathbf{k}\right)  =\widehat{\theta
}\left(  \ell\mathbf{k}\right)  $, we get the result.
\end{proof}

\begin{corollary}\label{cor:fgf}
In addition, assume $\theta$ is a smooth function with $\widehat{\theta
}\left(  \mathbf{k}\right)  =\widehat{\theta}\left(  \left\vert \mathbf{k}%
\right\vert \right)  $, let $f_{L}$ be the probability density of $L$ and
assume $\Gamma$ is a function of $L$: $\Gamma=\gamma\left(  L\right)  $.
Assume
\[
\gamma^{2}\left(  r\right)  f_{L}\left(  r\right)  =Cr^{\alpha}%
\]
for some $C>0$ and%
\[
\alpha>-1.
\]
Call
\[
D:=\int_{0}^{\infty}\left\vert \widehat{\theta}\left(  r\right)  \right\vert
^{2}\gamma^{2}\left(  r\right)  f_{L}\left(  r\right)  dr
\]
which is a finite constant. Then%
\[
Q_{\text{vortex}}\left(  x\right)  =D\sum_{\mathbf{k\in}\mathbb{Z}_{0}^{2}%
}\frac{1}{\left\vert \mathbf{k}\right\vert ^{3+\alpha}}P_{\mathbf{k}%
}e^{i\mathbf{k}\cdot x}.
\]
This is the covariance function of a SFGF of index%
\[
s=\frac{3+\alpha}{2}.
\]

\end{corollary}

\begin{proof}
Since $\theta$ is smooth, $\widehat{\theta}\left(  r\right)  $ has a fast
decay which makes $\widehat{\theta}\left(  r\right)  r^{\alpha}$ integrable at
infinity for every $\alpha$; and it is also integrable at zero because
$\alpha>-1$. From the assumptions,%
\begin{align*}
\mathbb{E}\left[  \Gamma^{2}\left\vert \widehat{\theta}\left(  L\mathbf{k}%
\right)  \right\vert ^{2}\right]    & =\mathbb{E}\left[  \gamma^{2}\left(
L\right)  \left\vert \widehat{\theta}\left(  \left\vert L\mathbf{k}\right\vert
\right)  \right\vert ^{2}\right]  \\
& =\int_{0}^{\infty}\gamma^{2}\left(  \ell\right)  \left\vert \widehat{\theta
}\left(  \ell\left\vert \mathbf{k}\right\vert \right)  \right\vert ^{2}%
f_{L}\left(  \ell\right)  d\ell\\
& =\left\vert \mathbf{k}\right\vert ^{-1}\int_{0}^{\infty}\left\vert
\widehat{\theta}\left(  r\right)  \right\vert ^{2}\gamma^{2}\left(  \left\vert
\mathbf{k}\right\vert ^{-1}r\right)  f_{L}\left(  \left\vert \mathbf{k}%
\right\vert ^{-1}r\right)  dr\\
& =\left\vert \mathbf{k}\right\vert ^{-1-\alpha}D.
\end{align*}

\end{proof}

Notice that $\alpha>-1$ corresponds to%
\[
s>1
\]
so the SGFF ($s=1$) is a (just excluded) limit case.

Recall that the solenoidal Kraichnan model with scaling parameter $\zeta$ is
defined, on the torus $\mathbb{T}^{d}$, by the covariance function
\[
Q_{\text{Kraichnan}}\left(  x\right)  =D\sum_{\mathbf{k\in}\mathbb{Z}_{0}^{d}%
}\frac{1}{\left\vert \mathbf{k}\right\vert ^{d+\zeta}}P_{\mathbf{k}%
}e^{i\mathbf{k}\cdot x}.
\]
We see thus that the vortex noise, in dimension $d=2$ (see next section for
$d=3$), covers Kraichnan model with scaling parameter
\[
\zeta=1+\alpha>0
\]
(any positive $\zeta$ is covered). 

The space-scale $\ell$ of the vortices is free in the previous results. If we
restrict ourselves to small vortices, namely we take $f_{L}\left(  r\right)
=0$ for $r>k_{0}^{-1}$ we get:

\begin{corollary}\label{cor:fractal}
Under the same assumptions of the previous corollary except for%
\[
\gamma^{2}\left(  r\right)  f_{L}\left(  r\right)  =Cr^{\alpha}1_{\left\{
r\leq k_{0}^{-1}\right\}  }%
\]
for some $C,k_{0}>0$ and $\alpha>-1$, we get
\[
Q_{\text{vortex}}\left(  x\right)  =\frac{1}{k_{0}^{3+\alpha}}\sum_{\left\vert
\mathbf{k}\right\vert >k_{0}}\frac{D\left(  \left\vert \mathbf{k}\right\vert
/k_{0}\right)  }{\left(  \left\vert \mathbf{k}\right\vert /k_{0}\right)
^{3+\alpha}}P_{\mathbf{k}}e^{i\mathbf{k}\cdot x}+R_{k_{0}}\left(  x\right)
\]
where%
\[
\lim_{\kappa\rightarrow\infty}D\left(  \kappa\right)  =D
\]%
\[
\left\Vert R_{k_{0}}\left(  x\right)  \right\Vert \leq\frac{C^{\prime}}%
{\alpha+1}\frac{\log k_{0}}{k_{0}^{1+\alpha}}%
\]
for some constant $C^{\prime}>0$. 
\end{corollary}

\begin{proof}
As above,%
\[
\mathbb{E}\left[  \Gamma^{2}\left\vert \widehat{\theta}\left(  L\mathbf{k}%
\right)  \right\vert ^{2}\right]  =\left\vert \mathbf{k}\right\vert
^{-1-\alpha}D\left(  \left\vert \mathbf{k}\right\vert /k_{0}\right)  .
\]
The first limit property is obvious. Moreover (using also $\left\Vert
\widehat{\theta}\right\Vert _{\infty}\leq1$)%
\[
D\left(  \kappa\right)  \leq C\frac{\kappa^{\alpha+1}}{\alpha+1}%
\]
hence%
\[
\frac{1}{k_{0}^{3+\alpha}}\sum_{\left\vert \mathbf{k}\right\vert \leq k_{0}%
}\frac{D\left(  \left\vert \mathbf{k}\right\vert /k_{0}\right)  }{\left(
\left\vert \mathbf{k}\right\vert /k_{0}\right)  ^{3+\alpha}}\leq\frac
{C}{\alpha+1}\frac{1}{k_{0}^{1+\alpha}}\sum_{\left\vert \mathbf{k}\right\vert
\leq k_{0}}\left\vert \mathbf{k}\right\vert ^{-2}\leq\frac{C^{\prime}}%
{\alpha+1}\frac{1}{k_{0}^{1+\alpha}}\log k_{0}.
\]

\end{proof}

We thus see that, up to lower order terms, the vortex model with cut-off
corresponds to Kraichnan model with infrared cut-off $k_{0}$ (cf. \cite[eq. (2.3)]{Eyink-2}).

Finally, we remark that the model has the flexibility of multifractality. To
explain it in the simplest possible case, assume
\begin{align*}
\gamma^{2}\left(  r\right)  f\left(  r\right)   &  =\sum_{i=1}^{N}%
C_{i}r^{\alpha_{i}}\\
D_{i} &  :=\int_{0}^{\infty}\left\vert \widehat{\theta}\left(  r\right)
\right\vert ^{2}C_{i}r^{\alpha_{i}}dr.
\end{align*}
Then we get%
\[
Q_{\text{vortex}}\left(  x\right)  =\sum_{i=1}^{N}D_{i}\sum_{\mathbf{k}}%
\frac{1}{\left\vert \mathbf{k}\right\vert ^{3+\alpha_{i}}}P_{\mathbf{k}%
}e^{i\mathbf{k}\cdot x}.
\]
Clearly one can do the same with a continuously distributed multifractality in
place of the finite sum (we void to introduce additional notations to explain
this point).

\begin{remark}
An intriguing but extremely difficult question (we thank an anonymous referee
for it) is whether we may infer the value of the scaling exponent $\zeta$ of
Kraichnan model, or a multifractal version of it, from the similarity with the
vortex noise. It was the main aim of the outstanding book \cite{Chorin}, which
- as admitted by the author - remained open at the time of the book and it is
still open now. Two examples of attempts in this direction have been
\cite{Lions-Majda} and \cite{FG}; in the latter work a multifractal formalism based
on vortex filaments was developed. However, it must be stressed that no one of
these works deduced K41 or other scalings from vortex models; they could only
reproduce scalings chosen a priori.
\end{remark}

\subsection{Covariance of 3D vortex noise}

Next, we turn to the 3D case, and recall the definition of the noise based on vortex filaments \eqref{def-3d}. The covariance of the noise is given by
\[
\left\langle \Q v,w\right\rangle =\mathbb{E}\left[  \Gamma^{2}\int_{\T^3} K_{L}\left(
x,X_{\cdot}\right)  \cdot v\left(  x\right)  dx\int_{\T^3} K_{L}\left(  x^{\prime}%
,X_{\cdot}\right)  \cdot w\left(  x^{\prime}\right)  dx^{\prime}\right],
\]
where 
\[
\int_{\T^3} K_{L}\left(
x,X_{\cdot}\right)  \cdot v\left(  x\right) dx=\int_{\T^3}\int_{\T^3}  v(x)\cdot K(x,y)\times \int_0^U\theta_L(y-X_t)\, dX_t\, dxdy.
\]
For simplicity we set from now on the time-horizon $U=1$, and assume that the 3D Brownian motion $(X_t)$ starts from uniform distribution on $\T^3$, hence for any time $t>0$, the distribution of $X_t$ remains uniform. $(X_t)$ is also independent of $(\Gamma, L)$. Using vector identity we may rewrite
\begin{align*}
\int_{\T^3} K_{L}\left(
x,X_{\cdot}\right)  \cdot v\left(  x\right)dx &=\int_{\T^3} \int_{\T^3} \int_0^1 \theta_L(y-X_t) v(x)\times K(x,y)\cdot dX_tdxdy\\
&=\int_0^1 \left[\theta_L^T * \left(\int_{\T^3} v(x)\times K(x-\cdot )dx \right)\right](X_t) \, \cdot dX_t.
\end{align*}
For the 3D kernel $K$ \eqref{kernel-3d}, we still have the property that $K\left(
x,a\right)  =K\left(  x-a\right)  =-K\left(  a-x\right)  $. 


Our first result is that in 3D the vortex noise has the same covariance structure as in 2D case.
\begin{proposition}
Assume $\theta$ symmetric, and $(X_{t})$ independent of $\left(  \Gamma
,L\right)$ and starts from uniform distribution on $\T^3$. Then%
\begin{equation*}
Q_{\text{vortex}}\left(  x\right)  =\sum_{\mathbf{k\in}\mathbb{Z}_{0}^{3}%
}\mathbb{E}\left[  \Gamma^{2}\left\vert \widehat{\theta}\left(  L\mathbf{k}%
\right)  \right\vert ^{2}\right]  \frac{1}{\left\vert \mathbf{k}\right\vert
^{2}}P_{\mathbf{k}}e^{i\mathbf{k}\cdot x}.%
\end{equation*}
\end{proposition}
\begin{proof}
\begin{align*}
&\langle \Q v, w\rangle =\E\Big[\Gamma^2\int_0^1 \left[\theta_L * \left(\int_{\T^3} v(x)\times K(x-\cdot )dx \right)\right](X_t) \, \cdot dX_t\\
&\quad \quad\quad\quad\quad\quad   \int_0^1 \left[\theta_L * \left(\int_{\T^3} w(x')\times K(x'-\cdot )dx' \right)\right](X_t) \, \cdot dX_t\Big]\\
&=\E\left[ \Gamma^2\int_0^1 \left[\theta_L * \left(\int_{\T^3} v(x)\times K(x-\cdot )dx \right)\right](X_t)\cdot\left[\theta_L * \left(\int_{\T^3} w(x')\times K(x'-\cdot )dx' \right)\right](X_t)\, dt\right]\\
&=\E\left[ \Gamma^2\int_{\T^3}\left[\theta_L * \left(\int_{\T^3} v(x)\times K(x-\cdot )dx \right)\right](z)\cdot\left[\theta_L * \left(\int_{\T^3} w(x')\times K(x'-\cdot )dx' \right)\right](z)\,dz\right],
\end{align*}
where we take conditional expectation with respect to $(X_t)$ first, using its time-stationarity and uniform distribution, whereas the randomness of $(\Gamma, L)$ remains.

By Parseval theorem  and vector identities, we may rewrite
\begin{align*}
\langle \Q v, w\rangle &=\E\left[\Gamma^2\sum_{\bk\in\Z^3_0} \wh{\theta_L}(\bk)\Big(\int_{\T^3} v(x)\times K(x-\cdot )dx\Big)^\wedge(\bk)\cdot \ovl{\wh{\theta_L}(\bk)\Big(\int_{\T^3} w(x)\times K(x-\cdot )dx\Big)^\wedge(\bk)}\right]\\
&=\sum_{\bk\in\Z^3_0}\E\left[\Gamma^2\left|\wh{\theta_L}(\bk)\right|^2\big(\wh{v}(\bk)\times \wh{K}(\bk)\big)\cdot \big(\ovl{\wh{w}(\bk)}\times \ovl{\wh{K}(\bk)}\big)\right]\\
&=\sum_{\bk\in\Z^3_0}\E\left[\Gamma^2\left|\wh{\theta_L}(\bk)\right|^2\ovl{\wh{w}(\bk)}\cdot\Big(\ovl{\wh{K}(\bk)}\times \big(\wh{v}(\bk)\times \wh{K}(\bk)\big)\Big)\right].
\end{align*}
By properties of triple cross product, we have that 
\begin{align*}
\ovl{\wh{K}(\bk)}\times \big(\wh{v}(\bk)\times \wh{K}(\bk)\big)
&=\wh{v}(\bk)\left(\wh{K}(\bk)\cdot\ovl{\wh{K}(\bk)}\right)-\wh{K}(\bk)\left(\ovl{\wh{K}(\bk)}\cdot\wh{v}(\bk)\right),
\end{align*}
hence 
\begin{align*}
&\ovl{\wh{w}(\bk)}\cdot\Big(\ovl{\wh{K}(\bk)}\times \big(\wh{v}(\bk)\times \wh{K}(\bk)\big)\Big) \\
&=\left|\wh{K}(\bk)\right|^2\big(\wh{v}(\bk)\cdot\ovl{\wh{w}(\bk)}\big) - \big(\ovl{\wh{w}(\bk)}\cdot\wh{K}(\bk)\big)\big(\ovl{\wh{K}(\bk)}\cdot\wh{v}(\bk)\big)\\
&=\left|\wh{K}(\bk)\right|^2\big(\wh{v}(\bk)\cdot\ovl{\wh{w}(\bk)}\big) - \ovl{\wh{w}(\bk)}^T\big(\wh{K}(\bk)\otimes\ovl{\wh{K}(\bk)}\big)\wh{v}(\bk)\\
&=\frac{1}{|\bk|^2}\big(\wh{v}(\bk)\cdot\ovl{\wh{w}(\bk)}\big)-\ovl{\wh{w}(\bk)}^T\big(\frac{\bk}{|\bk|^2}\otimes\frac{\bk}{|\bk|^2}\big)\wh{v}(\bk)\\
&=\frac{1}{|\bk|^2}\left\langle\left(I-\frac{\bk}{|\bk|}\otimes\frac{\bk}{|\bk|}\right)\wh{v}(\bk), \ovl{\wh{w}(\bk)}\right\rangle,
\end{align*}
recalling that in 3D,
\[
\wh{K}(\bk)=i\frac{\bk}{|\bk|^2}.
\]
Thus, we may conclude that 
\begin{align*}
\langle \Q v, w\rangle &=\sum_{\bk\in\Z^3_0}\E\left[\Gamma^2\left|\wh{\theta_L}(\bk)\right|^2\right]\frac{1}{|\bk|^2}\left\langle P_{\mathbf{k}%
}\widehat{v}\left(  \mathbf{k}\right)  ,\overline{\widehat{w}\left(
\mathbf{k}\right)  }\right\rangle,
\end{align*}
where $P_\bk=I-\frac{\bk}{|\bk|}\otimes\frac{\bk}{|\bk|}$ is the projector on the orthogonal to $\bk$. This yields in turn that the covariance matrix of the noise is given by
\[
Q_{\text{vortex}}\left(  x,x^{\prime}\right) =\sum_{\mathbf{k}\in\Z^3_0}\mathbb{E}\left[  \Gamma^{2}\left\vert \widehat{\theta
_{L}}\left(  \mathbf{k}\right)  \right\vert ^{2}\right]  \frac
{1}{\left\vert \mathbf{k}\right\vert ^{2}}P_{\mathbf{k}}e^{i\mathbf{k}%
\cdot\left(  x-x^{\prime}\right)  }.
\]
\end{proof}

This formula agrees with the formula \eqref{cov-2d} obtained for 2D, hence Corollary \ref{cor:fgf} applies in 3D without change (except for summation over $\bk\in\Z^3_0$). 

Our result in 3D covers Kraichnan noise with parameter 
\[
\zeta=\alpha>-1.
\]
We can also restrict the vortices to small scales by introducing a cutoff $k_0$, as in Corollary \ref{cor:fractal}. Here, we need to restrict to $\alpha>0$ in its statement, so that the remainder $R_{k_0}(x)$ is of lower order:
\[
\left\Vert R_{k_{0}}\left(  x\right)  \right\Vert \leq\frac{C^{\prime}}%
{\alpha+1}\frac{1}{k_{0}^{\alpha}}.%
\]

\section{The effect of vortex structure noise on passive scalars}

\subsection{Introduction}

Regarding eddy diffusion enhancement in domains with boundary, we recall the
following theorem proved in \cite[Theorems 1.1, 1.3]{FGL}. Here, we have a
passive scalar $T$ driven by the white-in-time, correlated-in-space noise $\partial_tW$ produced by our vortex
structures, where $W(t,x)$ is the limit Gaussian process obtained via the invariance principle in Theorem \ref{thm:inv-prin}:
\begin{align*}
\partial_{t}T+\partial_{t} W\circ\nabla T = \kappa\Delta T,
\end{align*}
$\circ$ denotes Stratonovich integration, and scalar $\kappa>0$. We denote the smallest eigenvalue of the matrix $Q(x,x)$ by
\begin{align*}
q(x,x)  &  := \min_{0\neq \xi\in\R^d}\frac{\xi^{T} Q(x,x)\xi}{\xi^{T}\xi},
\end{align*}
and the squared operator norm $\|Q^{1/2}\|_{L^2(\D)\to L^2(\D)}^2$ by
\begin{align*}
\epsilon_{Q}  &  := \sup_{0\neq v\in H}\frac{\int_{\D}\int_{\D}v^{T}(x)Q(x,y)v(y)dxdy}%
{\int_{\D}v(x)^{T}v(x)dx}.
\end{align*}

\begin{theorem}\label{thm:eddy}
{{\cite[Theorems 1.1, 1.3]{FGL}}}
\newline(a). 
For any $T_0\in H$ measurable, and any $t\ge0$, we have that
\[
\mathbb{E}\left[  \left(\int_{\D}|T(t,x)|dx\right)^2\right]  \le\left(
\frac{\epsilon_{Q}}{\kappa}+2|\D|e^{-2t\lambda_{\D,\kappa,Q}}\right)
\mathbb{E}\left[  ||T_{0}||^{2}_{L^{2}}\right] ,
\]
where $\lambda_{\D,\kappa,Q}$ is the first eigenvalue of the elliptic operator $-A_Q$, for 
\[
A_Q:=\kappa\Delta+\frac{1}{2}\text{div}(Q(x,x)\nabla \cdot).
\]
(b). There exists a constant $C_{\D,d}>0$ such that
\[
\lambda_{\D,\kappa,Q}\ge C_{\D,d}\min\left(  \sigma^{2}, \kappa/\delta\right)
,
\]
for every $Q$ such that
\[
\inf_{x\in \D_\delta}q(x,x)\ge\sigma^{2}.
\]
\end{theorem}

In view of this theorem, our aim is to show that the noises based on vortex structures in 2D and 3D that we constructed in Sections \ref{sec:jump-noise}-\ref{sec:noise}, for small $L$, enjoy the property that they have small $\ep_Q$ and large $q(x,x)$, simultaneously, once the other parameters of the model are tuned properly. Here, we assume that $\Gamma, U, L, X_\cdot$ are independent.

For technical reasons, we demonstrate this only for the torus $\D=\T^d$, $d=2,3$. The same conclusions should hold true for any regular domains $\D$, but the corrector part of the Green function $h_{\D,x_0}(x)$ is difficult to handle, hence we prefer to state in the simple case of torus. Note in this case we do not have a boundary hence $\D_\delta=\D$, $\delta=0$, and we can put the stopping time $\tau=\infty$ in the 3D case.

\subsection{The 2D case}

The following theorem applies to any realization $\ell$ of $L$. For fixed $\ell>0$, we shall use (recall \eqref{def-2d})
\begin{align*}
Q_\ell\left(  x,y\right)   &  =\mathbb{E}\left[  \Gamma ^{2}K_{\ell}%
(x,X_{0})\otimes K_{\ell}(y,X_{0})\right].
\end{align*}
Therefore we have, for $\xi\in\mathbb{R}^{2}$%
\[
\xi^{T}Q_\ell\left(  x,x\right) \xi=\mathbb{E}\left[  \Gamma ^{2}\left\vert
K_{\ell}(x,X_{0})\cdot \xi\right\vert ^{2}\right]  
\]
while for $v\in H$
\[
\left\langle \mathbb{Q}_\ell v,v\right\rangle =\int_{\T^2}\int_{\T^2}v\left(  x\right)
^{T}Q_\ell \left(  x,y\right)  v\left(  y\right)  dxdy=%
\mathbb{E}\left[  \Gamma ^{2}\left(  \int_{\T^2}v\left(  x\right)  \cdot K_{\ell}%
(x,X_{0})dx\right)  ^{2}\right]  .
\]
In the next statement we set $\sigma^2=\E(\Gamma^2)$.
\begin{theorem}\label{thm-2D}
i) There exists a finite constant $C$ such that for every $v\in H$ and $\ell\in(0,1)$,
\[
\frac{\left\langle \mathbb{Q}_{\ell}v,v\right\rangle}{\left\Vert v\right\Vert _{H}^{2}} \leq C\sigma^2.%
\]

ii) 
There exists a constant $c>0$ 
such
that for all $x\in\T^2$, $v\in\mathbb{R}^{2}$ and $\ell\in(0,1)$,
\[
\frac{v^{T}Q_{\ell}\left(  x,x\right)  v}{\left\vert
v\right\vert ^{2}} \ge c\;\sigma^2|\log\ell|.
\]
\end{theorem}

\begin{remark}
We can choose $\sigma^2=\E(\Gamma^2)$ to be small, then choose $\ell$ small enough such that $\sigma^2|\log\ell|$ is large, to fulfill the conditions in Theorem \ref{thm:eddy}.
\end{remark}

\begin{proof}
Since $\D=\T^2$, the function $\nabla
_{x}^{\perp}h_{\D}\left(  x,y\right)  $ is bounded above uniformly and does not affect the computations on $K\left(  x,y\right)$, which essentially can be based only on the term
$\frac{1}{2\pi}\frac{\left(  x-y\right)  ^{\perp}}{\left\vert x-y\right\vert
^{2}}$. Thus, 
we use
the approximation, for all $x\in \T^2$, a.s.,%
\[
\left\vert K_{\ell}(x,X_{0})\right\vert \leq\int_{\T^2}\left\vert K\left(
x,y\right)  \right\vert \theta_{\ell}(y-X_{0})dy\sim\frac{1}{2\pi}\int%
_{\T^2}\frac{1}{\left\vert x-y\right\vert }\theta_{\ell}(y-X_{0})dy.
\]
Let $C_{K_{\ell}}$ be the random variable defined as
\[
C_{K_{\ell}}:=\int_{\T^2}\left\vert K_{\ell}(x,X_{0})\right\vert dx
\]
Under our approximation we have%
\begin{align*}
\int_{\T^2 }\left\vert K_{\ell}(x,X_{0})\right\vert dx  &  \le \frac{1}{2\pi}%
\int_{\T^2 }\int_{\T^2 }\frac{1}{\left\vert x-y\right\vert }\theta_{\ell}%
(y-X_{0})dydx +C_1\\
&  =\frac{1}{2\pi}\int_{\T^2 }\left(  \int_{\T^2 }\frac{1}{\left\vert x-y\right\vert
}dx\right)  \theta_{\ell}(y-X_{0})dy +C_1\\
&  \leq C\int_{\T^2 }\theta_{\ell}(y-X_{0})dy+C_1=C+C_1,
\end{align*}
hence $C_{K_{\ell}}$ is finite a.s. and even uniformly bounded above. Then%
\begin{align*}
\left\langle \mathbb{Q}_\ell v,v\right\rangle  &  \leq%
\mathbb{E}\left[  \Gamma ^{2}\left(  \int_{\T^2 }\left\vert v\left(  x\right)  \right\vert
\left\vert K_{\ell}(x,X_{0})\right\vert dx\right)  ^{2}\right] \\
&  \leq\mathbb{E}\left[ \Gamma ^{2} C_{K_{\ell}}^{2}\left(  \int%
_{\T^2 }\left\vert v\left(  x\right)  \right\vert \frac{\left\vert K_{\ell
}(x,X_{0})\right\vert }{C_{K_{\ell}}}dx\right)  ^{2}\right] \\
&  \leq\mathbb{E}\left[  \Gamma ^{2}C_{K_{\ell}}^{2}\int_{\T^2 }\left\vert
v\left(  x\right)  \right\vert ^{2}\frac{\left\vert K_{\ell}(x,X_{0}%
)\right\vert }{C_{K_{\ell}}}dx\right] \\
&  =\mathbb{E}\left[ \Gamma ^{2} C_{K_{\ell}}\int_{\T^2 }\left\vert v\left(
x\right)  \right\vert ^{2}\left\vert K_{\ell}(x,X_{0})\right\vert dx\right]
\end{align*}
Let $\widetilde{C}_{K_{\ell}}$ be the deterministic constant
defined as
\[
\widetilde{C}_{K_{\ell}}:=\sup_{x\in \T^2 }\mathbb{E}\left[ \Gamma^2 C_{K_{\ell}%
}\left\vert K_{\ell}(x,X_{0})\right\vert \right]  <\infty.
\]
We have proved%
\[
\left\langle \mathbb{Q}_\ell v,v\right\rangle \leq \widetilde{C}%
_{K_{\ell}}\left\Vert v\right\Vert _{H}^{2}.
\]
Concerning the size of $\widetilde{C}_{K_{\ell}}$, we have, under the assumptions that $p_{0}$ has a bounded
density%
\begin{align*}
\widetilde{C}_{K_{\ell}}  &  \leq C\sup_{x\in \T^2 }\mathbb{E}\left[  \Gamma^2\left\vert
K_{\ell}(x,X_{0})\right\vert \right] \\
&  \sim\frac{C}{2\pi}\sup_{x\in \T^2 }\mathbb{E}\left[\Gamma^2  \int_{\T^2 }\frac
{1}{\left\vert x-y\right\vert }\theta_{\ell}(y-X_{0})dy\right]\\
& =\frac{C}{2\pi}\sup_{x\in \T^2 }\E \left[\Gamma^2  \int_{\T^2 }\int\frac
{1}{\left\vert x-y\right\vert }\theta_{\ell}(y-x_{0})p_0(x_0)dx_0dy\right]\\
&\le \frac{Cp_{max}}{2\pi}\sup_{x\in \T^2 }\E \left[\Gamma^2  \int_{\T^2 }\frac
{1}{\left\vert x-y\right\vert }dy\right]\\
&  \leq C(p_{max},\T^2)\E (\Gamma^2)
\end{align*}
since 
\[
\sup_{x\in \T^2}\int_{\T^2 }\frac
{1}{\left\vert x-y\right\vert }dy\le C_{\T^2}.
\]
Therefore
\[
\left\langle \mathbb{Q}_\ell v,v\right\rangle \leq C\E (\Gamma ^{2})\left\Vert
v\right\Vert _{H}^{2}.
\]
This quantity is small if $\E (\Gamma ^{2})$ is small.

\smallskip
Concerning $v^{T}Q\left(  x,x\right)  v$, $v\in\mathbb{R}^{2}$, we have, using
again the simplified asymptotics,
\begin{align*}
v^{T}Q\left(  x,x\right)  v  &  =\E \left(\Gamma ^{2}\int_{\T^2}\left\vert K_{\ell
}(x,x_{0})\cdot v\right\vert ^{2}p_{0}\left(  dx_{0}\right) \right)\\
&  \sim\E \left(\frac{\Gamma ^{2}}{\left(  2\pi\right)  ^{2}}\int_{\T^2%
}\left\vert \int_{\T^2}\frac{\left(  x-y\right)  ^{\perp}\cdot
v}{\left\vert x-y\right\vert ^{2}}\theta_{\ell}(y-x_{0})dy\right\vert
^{2}p_{0}\left(  dx_{0}\right)\right)  .
\end{align*}
Given any $x\in \T^2$ and unit vector $v\in\mathbb{R}^{2}$,
there is a cone $C\left(  x,v\right)  \subset \T^2$ (a set of the form $x+rw$,
$r\in\left[  0,r_{0}\right]  $, $\left\vert w\right\vert =1$, $w\cdot
e\geq\alpha$ for some $\left\vert e\right\vert =1$ and $\alpha\in\left(
0,1\right)  $) such that
\[
\left(  x-x_{0}\right)  ^{\perp}\cdot v\geq\frac{1}{2}\left\vert
x-x_{0}\right\vert \left\vert v\right\vert \text{ for every }x_{0}\in C\left(
x,v\right)
\]
and
\[
\left\vert C\left(  x,v\right)  \right\vert \geq\eta>0.
\]

Moreover, assume $p_{0}\left(  dx_{0}\right)  $ is bounded below by $p_{\min
}Leb$ for some constant $p_{\min}>0$. We then have%
\[
v^{T}Q_\ell \left(  x,x\right)  v\geq \E \left(\frac{\Gamma ^{2}p_{\min}}{\left(
2\pi\right)  ^{2}}\int_{C\left(  x,v\right)  }\left\vert \int_{B\left(
x_{0},\ell\right)  }\frac{\left(  x-y\right)  ^{\perp}\cdot v}{\left\vert
x-y\right\vert ^{2}}\theta_{\ell}(y-x_{0})dy\right\vert ^{2}dx_{0}\right).
\]
Taking $\ell>0$ very small, reduce the cone $C\left(  x,v\right)  $ to the
set
\[
C_{\ell}\left(  x,v\right)  \subset C\left(  x,v\right)
\]
of points $x_{0}$ such that
\[
\text{dist}\left(  x_{0},\partial C\left(  x,v\right)  \right)  \geq 2\ell.
\]
We then have%
\[
y\in C\left(  x,v\right)  \text{ if }y\in B\left(  x_{0},\ell\right)  \text{
with }x_{0}\in C_{\ell}\left(  x,v\right)
\]
and thus
\begin{align*}
v^{T}Q_\ell\left(  x,x\right)  v  &  \geq\E \left(\frac{\Gamma ^{2}p_{\min}}{\left(
2\pi\right)  ^{2}}\int_{C_{\ell}\left(  x,v\right)  }\left\vert \int_{B\left(
x_{0},\ell\right)  }\frac{\frac{1}{2}\left\vert x-y\right\vert \left\vert
v\right\vert }{\left\vert x-y\right\vert ^{2}}\theta_{\ell}(y-x_{0}%
)dy\right\vert ^{2}dx_{0}\right)\\
& =\E \left(\frac{\Gamma ^{2}p_{\min}\left\vert v\right\vert ^{2}}{4\left(
2\pi\right)  ^{2}}\int_{C_{\ell}\left(  x,v\right)  }\left\vert \int_{B\left(
x_{0},\ell\right)  }\frac{1}{\left\vert x-y\right\vert }\theta_{\ell}%
(y-x_{0})dy\right\vert ^{2}dx_{0}\right)\\
&  =\E \left(\frac{\Gamma ^{2}p_{\min}\left\vert v\right\vert ^{2}}{4\left(
2\pi\right)  ^{2}}\int_{C_{\ell}\left(  x,v\right)  }\left(  \theta_{\ell }\ast\frac
{1}{\left\vert \cdot\right\vert }\right) ^{2}(x-x_0)dx_{0}\right)\\
&  =\E \left(\frac{\Gamma ^{2}p_{\min}\left\vert v\right\vert ^{2}}{4\left(
2\pi\right)  ^{2}}\int_{C_{\ell}\left(  0,v\right)  }\left(  \theta_{\ell }\ast\frac
{1}{\left\vert \cdot\right\vert }\right) ^{2}(x_0)dx_{0}\right)\\
&\ge c(p_{\text{min}},\eta)|v|^2\E\left(\Gamma^2\right)\int_{\T^2}\left(  \theta_{\ell }\ast\frac
{1}{\left\vert \cdot\right\vert }\right)  ^{2}\left(  x_0\right)  dx_0.
\end{align*}
The last inequality is because the quantity $x_0\mapsto \left(  \theta_{\ell }\ast\frac
{1}{\left\vert \cdot\right\vert }\right)  ^{2}\left(  x_0\right) $ is rotationally invariant, hence the integral $\int_{C_{\ell}\left(  0,v\right)  }\left(  \theta_{\ell }\ast\frac
{1}{\left\vert \cdot\right\vert }\right)  ^{2}\left(  x_0\right)  dx_0$ does not depend on $v$. Since $|C(0,v)|\ge\eta$, we have that 
\[
\int_{C_{\ell}\left(  0,v\right)  }\left(  \theta_{\ell }\ast\frac
{1}{\left\vert \cdot\right\vert }\right)  ^{2}\left(  x_0\right)  dx_0\ge c\eta^{-1} \int_{\T^2}\left(  \theta_{\ell }\ast\frac
{1}{\left\vert \cdot\right\vert }\right)  ^{2}\left(  x_0\right)  dx_0.
\]
Let us investigate the problem of the scaling
in $\ell $ of the quantity $\int_{\T^2}\left(  \theta_{\ell }\ast\frac
{1}{\left\vert \cdot\right\vert }\right)  ^{2}\left(  x\right)  dx$. Given the
mollifier $\theta_{\ell }\left(  x\right)  =\ell ^{-2}\theta\left(
\ell ^{-1}x\right)  $, that we assume the best possible one (nonnegative, smooth, symmetric), let us introduce the smooth symmetric pdf, compactly supported in $B(0,2)$,
\[
\theta^{\left(  2\right)  }\left(  z\right)  :=\int\theta\left(  z-z^{\prime
}\right)  \theta\left(  z^{\prime}\right)  dz^{\prime}.
\]
Then%
\begin{align*}
\theta_{\ell }^{\left(  2\right)  }\left(  z\right)    & =\ell 
^{-2}\theta^{\left(  2\right)  }\left(  \ell ^{-1}x\right)  =\int%
\ell ^{-2}\theta\left(  \ell ^{-1}z-z^{\prime}\right)  \theta\left(
z^{\prime}\right)  dz^{\prime}\\
& \overset{z^{\prime}=\ell ^{-1}w}{=}\int\ell ^{-2}\theta\left(
\ell ^{-1}\left(  z-w\right)  \right)  \theta\left(  \ell ^{-1}w\right)
\ell ^{-2}dw\\
& =\int\theta_{\ell }\left(  z-w\right)  \theta_{\ell }\left(  w\right)
dw\\
& =\left(  \theta_{\ell }\ast\theta_{\ell }\right)  \left(  z\right)  .
\end{align*}
Below we shall use the formula%
\[
\theta_{\ell }^{\left(  2\right)  }\left(  y-y^{\prime}\right)  =\int%
\theta_{\ell }\left(  x-y\right)  \theta_{\ell }\left(  x-y^{\prime
}\right)  dx
\]
true because%
\begin{align*}
\theta_{\ell }^{\left(  2\right)  }\left(  y-y^{\prime}\right)    &
=\int\theta_{\ell }\left(  y-y^{\prime}-w\right)  \theta_{\ell }\left(
w\right)  dw\\
& \overset{w=x-y^{\prime}}{=}\int\theta_{\ell }\left(  y-x\right)
\theta_{\ell }\left(  x-y^{\prime}\right)  dx
\end{align*}
(recall $\theta$ is symmetric).
After these preliminaries, we have
\[
\int\left(  \theta_{\ell }\ast\frac{1}{\left\vert \cdot\right\vert
}\right)  ^{2}\left(  x\right)  dx=\int\left(  \int\theta_{\ell }\left(
x-y\right)  \frac{1}{\left\vert y\right\vert }dy\right)  ^{2}dx
\]%
\begin{align*}
& =\int\int\int\theta_{\ell }\left(  x-y\right)  \theta_{\ell }\left(
x-y^{\prime}\right)  \frac{1}{\left\vert y\right\vert }\frac{1}{\left\vert
y^{\prime}\right\vert }dydy^{\prime}dx\\
& =\int\int\theta_{\ell }^{\left(  2\right)  }\left(  y-y^{\prime}\right)
\frac{1}{\left\vert y\right\vert }\frac{1}{\left\vert y^{\prime}\right\vert
}dydy^{\prime}\\
& =\int\int\theta_{\ell }^{\left(  2\right)  }\left(  z\right)  \frac
{1}{\left\vert y\right\vert }\frac{1}{\left\vert y-z\right\vert }dydz\\
& =\int\left(  \int\frac{1}{\left\vert y\right\vert }\frac{1}{\left\vert
y-z\right\vert }dy\right)  \theta_{\ell }^{\left(  2\right)  }\left(
z\right)  dz.
\end{align*}
Now we have to understand first the bahavior of
\[
z\mapsto\int\frac{1}{\left\vert y\right\vert }\frac{1}{\left\vert
y-z\right\vert }dy.
\]
We can prove that, for $\left\vert z\right\vert \leq1$, 
\begin{align*}
\int_{\R^2}\frac
{1}{\left\vert y\right\vert }\frac{1}{\left\vert y-z\right\vert }dy\geq
\left\vert \log\left\vert z\right\vert \right\vert .
\end{align*}
Indeed, since $|y-z|\le |y|+|z|$, 
\begin{align*}
&\int_{\R^2}\frac{1}{|y|}\frac{1}{|y-z|}dy \ge \int\frac{1}{|y|}\frac{1}{|y|+|z|}dy\\
& \ge \int_0^1\frac{1}{\rho}\frac{1}{\rho+|z|}\rho \; d\rho =\log(1+|z|)-\log|z|\\
&\ge -\log|z|=|\log|z||. 
\end{align*}
Then, $\int\left(  \int\frac{1}{\left\vert y\right\vert }\frac{1}{\left\vert
y-z\right\vert }dy\right)  \theta_{\ell }^{\left(  2\right)  }\left(
z\right)  dz$ can be bounded below by
\begin{align*}
\int\theta_{\ell }^{\left(  2\right)  }\left(  z\right)  \left\vert
\log\left\vert z\right\vert \right\vert dz
&\ge \ell ^{-2}\int_{\left\vert z\right\vert \leq \ell }\theta^{(2)}(\ell^{-1}z)\left\vert
\log\left\vert z\right\vert \right\vert dz \\
& \ge - c_\theta \ell ^{-2}\int%
_{0}^{\ell }r\log rdr\\
& =c_\theta\left(-\ell ^{-2}\left[  \frac{r^{2}}{2}\log r\right]  _{r=0}^{r=\ell 
}+\ell ^{-2}\int_{0}^{\ell }\frac{r^{2}}{2}\frac{1}{r}dr\right)\\
& =c_\theta\left(-\ell ^{-2}\frac{\ell ^{2}}{2}\log\ell + \ell ^{-2}%
\frac{\ell ^{2}}{4}\right)\\
& = c_\theta\left(\left\vert \log\ell \right\vert +\frac{1}{4}\right),
\end{align*}
where w.l.o.g.
\[
c_\theta:=\inf_{z\in B(0,1)}\theta^{(2)}(z)>0.
\]
This yields that  
\[
\frac{v^{T}Q_\ell\left(  x,x\right)  v}{|v|^2}\ge c \; \E (\Gamma^2)|\log\ell|
\]
for some $c>0$ and any $\ell\in(0,1)$.
\end{proof}

\subsection{The 3D case}

Recall that we take $\D=\T^3$, hence the computation below can be based solely on the $\frac{1}{4\pi}\frac{x-y}{\left\vert x-y\right\vert ^{3}}\times$ part of the kernel $K(x,y)$ \eqref{kernel-3d}, with the other part from $\nabla_xh_{\D,x_{0}}\left(
x\right)  \times$ uniformly bounded. We also set $\tau=\infty$. The following theorem applies to any realization $\ell$ of $L$. We shall use the notation $Q_\ell(x,y)$ and $\Q_\ell$ for fixed $\ell$, similarly to what is done in the 2D case, while recalling \eqref{def-3d}.

In the next statement we set $\sigma^2=\E(\Gamma^2)$.

\begin{theorem}\label{thm-3D}
i) There exists a constant $C<\infty$ such that for every $v\in H$ and $\ell\in(0,1)$,
\begin{align*}
\frac{\left\langle \mathbb{Q}_\ell v,v\right\rangle }{\left\Vert v\right\Vert
_{H}^{2}} \leq C\;\E (U)\sigma^2.
\end{align*}
ii) 
There exists a constant $c>0$ 
such
that for all $x\in\T^3$, $v\in\mathbb{R}^{3}$ and $\ell\in(0,1)$,
\begin{align*}
\frac{v^{T}Q_\ell\left(  x,x\right)  v}{\left\vert v\right\vert ^{2}}\geq
c\;\E (U)\sigma^2\ell^{-1}.
\end{align*}
\end{theorem}

\begin{remark}
We can choose the distribution of $(\Gamma,U)$ such that $\E(U)\sigma^2$ is small, then choose $\ell$ small enough such that $\E(U)\sigma^2\ell^{-1}$ is large, to fulfill the conditions in Theorem \ref{thm:eddy}.
\end{remark}

\begin{proof}
Taking any $v\in H$, we consider
\begin{align*}
\langle\mathbb{Q}_\ell v,v\rangle &  =\int_{\T^3}\int_{\T^3}v(x)^{T}Q_\ell (x,y)v(y)dxdy\\
&  =\mathbb{E}\left[  \Gamma ^{2} \left(  \int_{\T^3}v(x)\cdot \int_{\T^3}K\left(  x,y\right) \times \left(  \int_{0}^{U\wedge\tau}\theta
_{\ell}(y-X_{t})dX_{t}\right)  dydx\right)
^{2}\right]  .
\end{align*}
For any fixed realization of $(\Gamma , U)$, we take expectation over
$\mathcal{W}$ first
\begin{align*}
&  \langle\mathbb{Q}_\ell v,v\rangle \\
&  \sim\Gamma ^{2}\mathcal{W}\left[  \left(  \int_{\T^3}\int_{\T^3}\int_{0}%
^{U }\theta_{\ell}(y-X_{t})v(x)\cdot\frac{1}{4\pi}\frac
{x-y}{|x-y|^{3}}\times dX_{t}dydx\right)  ^{2}\right] \\
&  = \Gamma ^{2} \mathcal{W}\left[  \left(  \int_{0}^{U }\int%
_{\T^3}\int_{\T^3}\theta_{\ell}(y-X_{t})v(x)\times\frac{1}{4\pi}\frac{x-y}%
{|x-y|^{3}}dydx\cdot dX_{t}\right)  ^{2}\right] \\
&  =\Gamma ^{2} \mathcal{W}\left[  \int_{0}^{U }\left\vert
\int_{\T^3}\int_{\T^3}\theta_{\ell}(y-X_{t})v(x)\times\frac{1}{4\pi}\frac
{x-y}{|x-y|^{3}}dydx\right\vert ^{2}dt\right]
\end{align*}
where the last step is due to It\^{o} isometry. We further bound it above by
moving the norm inside the integral
\begin{align*}
&  \Gamma ^{2} \mathcal{W}\left[  \int_{0}^{U }\left(  \int%
_{\T^3}\int_{\T^3}\theta_{\ell}(y-X_{t})|v(x)|\frac{1}{4\pi}\frac{1}{|x-y|^{2}%
}dydx\right)  ^{2}dt\right] \\
&  =\Gamma ^{2} \mathcal{W}\left[  \int_{0}^{U }C_{u}^{2}\left(
\int_{\T^3}|v(x)|\frac{\int_{\T^3}\theta_{\ell}(y-X_{t})\frac{1}{4\pi}\frac
{1}{|x-y|^{2}}dy}{C_{u}}dx\right)  ^{2}dt\right] \\
&  \leq\Gamma ^{2} \mathcal{W}\left[  \int_{0}^{U }C_{u}\int%
_{\T^3}|v(x)|^{2}\int_{\T^3}\theta_{\ell}(y-X_{t})\frac{1}{4\pi}\frac{1}{|x-y|^{2}%
}dydxdt\right]
\end{align*}
where the random constant $C_{u}$
\[
C_{u}:=\int_{\T^3}\int_{\T^3}\theta_{\ell}(y-X_{t})\frac{1}{4\pi}\frac{1}{|x-y|^{2}%
}dydx\leq C_{\T^3}
\]
for some deterministic finite constant $C_{\T^3}$ (integrate first $dx$ then
$dy$). Set
\begin{align*}
C_{u}^{\prime}  &  :=\sup_{x\in \T^3}\mathcal{W}\left[  \int%
_{0}^{U }\int_{\T^3}\theta_{\ell}(y-X_{t})\frac{1}{4\pi}\frac
{1}{|x-y|^{2}}dydt\right]  .
\end{align*}
Recall that $X_{0}$ has density $p_{0}(x)$ which is bounded above uniformly by
$p_{\text{max}}$. Since the heat semigroup is an $L^{\infty}$-contraction, the
density of $X_{t}$ at any later time $t$ is bounded above by $p_{\text{max}}$,
thus we have
\begin{align*}
C^{\prime}_{u}  &  \leq Up_{\text{max}}\sup_{x\in \T^3}\int%
_{\mathbb{T}^{3}}\int_{\T^3}\theta_{\ell}(y-z)\frac{1}{4\pi}\frac{1}{|x-y|^{2}%
}dydz\leq U C^{\prime}_{\T^3}%
\end{align*}
(integrating first $dz$ then $dy$), for some deterministic finite constant
$C_{\T^3}^{\prime}$. We conclude with
\begin{align*}
\langle\mathbb{Q}v,v\rangle &  \le\E \left(  \Gamma ^{2}C^{\prime}%
_{u}\right)  \|v\|_{H}^{2}\leq C^{\prime}_{\T^3}\sigma^2\E \left(  
U\right)  \left\Vert v\right\Vert _{H}^{2}.
\end{align*}

Taking now any unit vector $v\in\mathbb{R}^{3}$, we consider for any
$x\in\T^3$, the quantity
\begin{align*}
v^{T}Q(x,x)v  &  =\mathbb{E}\left[  \Gamma ^{2}\left|  v\cdot u(x)\right|
^{2}\right]  .
\end{align*}
We again fix any realization of $(\Gamma , U, \ell)$, and take expectation
over $\mathcal{W}$ first
\begin{align*}
&  \mathcal{W}\left[  \Gamma ^{2}\left|  v\cdot u(x)\right|  ^{2}\right] \\
&  =\Gamma ^{2}\mathcal{W}\left[  \left(  \int_{\T^3}\int_{0}^{U 
}\theta_{\ell}(y-X_{t})v\cdot\frac{1}{4\pi}\frac{x-y}{|x-y|^{3}}\times
dX_{t}dy\right)  ^{2}\right] \\
&  = \Gamma ^{2}\mathcal{W}\left[  \left(  \int_{0}^{U }\int%
_{\T^3}\theta_{\ell}(y-X_{t})v\times\frac{1}{4\pi}\frac{x-y}{|x-y|^{3}}dy\cdot
dX_{t}\right)  ^{2}\right] \\
&  = \Gamma ^{2}\mathcal{W}\left[  \int_{0}^{U }\left\vert
\int_{\T^3}\theta_{\ell}(y-X_{t})v\times\frac{1}{4\pi}\frac{x-y}{|x-y|^{3}%
}dy\right\vert ^{2}dt\right]
\end{align*}
where the last step is due to It\^{o} isometry. 

Since $\D=\T^3$ compact, the density of $X_t$, denoted $p_t(z)$, converges to the uniform distribution, hence it is not hard to see that there exists some $p_{\text{min}}>0$ independent of $t$ such that 
\[
p_{t}(z)\geq p_{\text{min}},\quad z\in\T^3
,\;t\in\lbrack0,U].
\]
Then, we can continue to bound below $\mathcal{W}\left[  \Gamma ^{2}\left|
v\cdot u(x)\right|  ^{2}\right]  $ by
\begin{align*}
&  \Gamma ^{2}\int_{0}^{U}\int_{\T^3}\left\vert \int%
_{\T^3}\theta_{\ell}(y-z)v\times\frac{1}{4\pi}\frac{x-y}{|x-y|^{3}}dy\right\vert
^{2}p_{t}(z)dz\\
&  \geq\Gamma ^{2}p_{\text{min}}U\int_{\T^3}\left\vert
\int_{\T^3}\theta_{\ell}(y-z)v\times\frac{1}{4\pi}\frac{x-y}{|x-y|^{3}%
}dy\right\vert ^{2}dz.
\end{align*}
For any $x\in\T^3$, there exist a cone $C(x,v)$ and a ball
$B=B(x^{\ast},\ell/2)\subset C(x,v)$ of radius
$\ell/2$ with center $x^{\ast}$ with $|x-x^{\ast}|=2\ell$, such that provided
$z\in B$, we have all the $y$ that contribute to the above integral be
contained in $B(x^{\ast},3\ell/2)$ and $\ell/2\le
|x-y|\le7\ell/2$, and on the other hand the orientation of the cone is chosen
such that $v\times(x-y)$ are roughly in the same direction for all the $y$.
This implies that for some absolute constant $c>0$ and any $z\in B$,
\[
\left\vert \int_{\T^3}\theta_{\ell}(y-z)v\times\frac{1}{4\pi}\frac{x-y}%
{|x-y|^{3}}dy\right\vert \geq c|v|\int_{\T^3}\theta_{\ell}(y-z)\ell^{-2}%
dy=c\ell^{-2}.
\]
Thus, we have that, upon squaring and using $|B|\asymp\ell^{3}$,
\[
v^{T}Q_\ell(x,x)v\geq c\, p_{\text{min}}\E \left(  \Gamma ^{2}U\int_{B}\ell
^{-4}dz\right)  =c\, p_{\text{min}}\E \left(  \Gamma ^{2}U\right) \ell^{-1} .
\]
This completes the proof.
\end{proof}

\section{Acknowledgments}
We thank an anonymous referee for the contribution to Section \ref{sec:covariance}, which was prepared after the advice to compare better
our model with the FGF.

The research of the first author is funded by the European Union (ERC, NoisyFluid, n. 101053472). Views and opinions expressed are however those of the authors only and do not necessarily reflect those of the European Union or the European Research Council. Neither the European Union nor the granting authority can be held responsible for them.


\begin{thebibliography}{99}                                                                                               %
\bibitem {Apolinario}G. B. Apolin\'{a}rio, G. Beck, L. Chevillard, I.
Gallagher, R. Grande, A linear stochastic model of turbulent cascades and
fractional fields,  \qquad arXiv:2301.00780, 2023.


\bibitem {BessaihCoghi}H. Bessaih, M. Coghi, F. Flandoli, Mean field limit of
interacting filaments and vector valued non-linear PDEs. {\it{J. Stat. Phys.}} 166
(2017), no. 5, 1276--1309.

\bibitem {Breit Fe Hof}D. Breit , E. Feireisl, M. Hofmanov\'{a},
{\it{Stochastically Forced Compressible Fluid Flows}}, De Gruyter, Berlin 2018.

\bibitem {Capasso}V. Capasso, F. Flandoli, On stochastic distributions and
currents. {\it{Math. Mech. Complex Syst.}} 4 (2016), no. 3-4, 373--406.

\bibitem {Chaves}M. Chaves, K. Gawedzki, P. Horvai, A. Kupiainen, M.
Vergassola, Lagrangian dispersion in Gaussian self-similar velocity ensembles,
{\it{J. Stat. Phys.}} 113 (2003), 643--692.

\bibitem {Chorin}A. J. Chorin, {\it{Vorticity and Turbulence}}, Applied Mathematical Sciences 103, Springer, Berlin 1994.

\bibitem {Ch78}P. L. Chow, Stochastic partial differential equations in
turbulence related problems. {\it{Probabilistic analysis and related topics}}, Vol.
1, pp. 1--43. Academic Press, New York, 1978.

\bibitem {DaPZa92}G. Da Prato, J. Zabczyk, {\it{Stochastic Equations in Infinite
Dimensions}}, Cambridge University Press, Cambridge 1992.

\bibitem {Eyink}G. L. Eyink, J. Xin, Existence and uniqueness of $L^2$-solutions
at zero-diffusivity in the Kraichnan model of a passive scalar. Preprint 1996 (https://arxiv.org/abs/chao-dyn/9605008).


\bibitem {Eyink-2}G. L. Eyink, J. Xin, Self-similar decay in the Kraichnan model
of a passive scalar, {\it{J. Stat. Phys.}} 100 (2000), 3-4, 679-741.

\bibitem {Fla08}F. Flandoli, An Introduction to 3D Stochastic Fluid Dynamics,
In G. Da Prato, and M. R\"{o}ckner, editors, {\it{SPDE in Hydrodynamic: Recent
Progress and Prospects}}, pp 51--150. Springer, Berlin 2008.

\bibitem {Fla11}F. Flandoli, {\it{Random Perturbation of PDEs and Fluid Dynamic
Models}}, Lecture Nnotes in Mathematics 2015, Springer, Berlin 2011.

\bibitem {FlaCPDE}F. Flandoli, Weak vorticity formulation of 2D Euler
equations with white noise initial condition, {\it{Comm. Partial Differential
Equations}} 43 (2018), no. 7, 1102-1149.


\bibitem {FGL}F. Flandoli, L. Galeati and D. Luo. Eddy heat exchange at the
boundary under white noise turbulence. \textit{{Phil. Trans. R. Soc. A.}} 380 (2022), 20210096.

\bibitem {FG}F. Flandoli and M. Gubinelli. Statistics of a vortex filament
model. \textit{{Electron. J. Probab.}} 10 (2005), 865-900.

\bibitem {FlaGubGiaTor}F. Flandoli, M. Gubinelli, M. Giaquinta, V. Tortorelli,
Stochastic currents, {\it{Stoch. Process. Appl.}} 115 (2005), no. 9, 1583--1601.

\bibitem {FlaGubRusso}F. Flandoli, M. Gubinelli, F. Russo, On the regularity
of stochastic currents, fractional Brownian motion and applications to a
turbulence model. {\it{Ann. Inst. Henri Poincar\'{e} Probab. Stat.}} 45 (2009), no.
2, 545--576.

\bibitem {Galeati}L. Galeati, On the convergence of stochastic transport
equations to a deterministic parabolic one, {\it{Stoch. Partial Differ. Equ. Anal.
Comput.}} 8 (2020), no. 4, 833-868.

\bibitem {Giaquinta}M. Giaquinta, G. Modica, J. Sou\v{c}ek, {\it{Cartesian currents
in the calculus of variations. I. Cartesian currents.}} Springer-Verlag, Berlin, 1998.

\bibitem {Grotto}F. Grotto, Stationary solutions of damped stochastic
2-dimensional Euler's equation. {\it{Electron. J. Probab.}} 25 (2020), Paper No. 69,
24 pp.

\bibitem {Hyt vol 1}T. Hyt\"{o}nen, J. van Neerven, M. Veraar, L. Weis,
{\it{Analysis in Banach Spaces. Volume I: Martingales and Littlewood-Paley Theory}},
Springer, Berlin 2016.

\bibitem {Krai}R. H. Kraichnan, Inertial ranges in two-dimensional turbulence.
{\it{The Physics of Fluids}}, 10 (1967), 7, 1417-1423.

\bibitem{Krai-3}
R. H. Kraichnan, Small-scale structure of a scalar field convected by turbulence. {\it{The Physics of Fluids}} {\bf{11}}
(1968), 945-953.

\bibitem {Kr}R. H. Kraichnan, Anomalous scaling of a randomly advected passive
scalar. {\it{Phys. Rev. Lett.}} 72 (1994).


\bibitem {Kuk Shi 2012}S. B. Kuksin, A. Shirikyan, {\it{Mathematics of
Two-Dimensional Turbulence}}, Cambridge University Press, Cambridge 2012.

\bibitem{Lions-Majda}
P.-L. Lions, A. Majda, Equilibrium statistical theory for nearly parallel vortex filaments. {\it{Commun. Pure Appl. Math.}} LIII (2000), 76-142.

\bibitem {Lodhia}A. Lodhia, S. Sheffield, X. Sun, S. S. Watson, Fractional
Gaussian fields: a survey, {\it{Probability Surveys}} 13 (2016), 1-56.

\bibitem {MarchioroPulvirenti}C. Marchioro, M. Pulvirenti,
\textit{Mathematical Theory of Incompressible Nonviscous Fluids}, Applied
Mathematical Sciences, 96. Springer-Verlag, New York, 1994.

\bibitem {Metivier}M. M\'etivier, \textit{{Stochastic Partial Differential
Equations in Infinite Dimensional Spaces}}, Quaderni Scuola Normale Superiore,
Pisa 1988.

\bibitem {PrRo07}C. Pr\'{e}v\^{o}t, M. R\"{o}ckner, \textit{A concise course
on stochastic partial differential equations}, Lecture Notes in Mathematics,
1905, Springer, Berlin, 2007.

\bibitem {Rebolledo}R. Rebolledo, La m\'{e}thode des martingales appliqu\'{e}e
\`a la convergence en loi des processus, \textit{{M\'{e}moires de la S.F.M.}},
t. 62, 1979.

\bibitem {RozLot}B. L. Rozovsky, S. Lototsky, {\it{Stochastic Evolution Systems,
Linear Theory and Applications to Non-Linear Filtering}}, Springer, Berlin 2018.

\bibitem {Temam1}R. Temam, {\it{Navier-Stokes Equations}}, North--Holland Pub.
Company, in English, 1977.

\bibitem {ViFu88}M. J. Vishik, A. V. Fursikov, {\it{Mathematical Problems in
Statistical Hydromechanics}}, Kluwer, Boston, 1988.
\end{thebibliography}
\end{document}